\documentclass[12pt]{article}

\usepackage{amsmath,amssymb,amsthm}
\theoremstyle{definition} 

\newtheorem{theorem}{Theorem}[section]
\newtheorem{proposition}[theorem]{Proposition}
\newtheorem{lemma}[theorem]{Lemma}
\newtheorem{corollary}[theorem]{Corollary}
\newtheorem{conjecture}[theorem]{Conjecture}

\theoremstyle{definition}

\newtheoremstyle{myexample}
  {\medskipamount} 
  {0pt}            
  {\normalfont}    
  {0pt}            
  {\normalfont}    
  {.}              
  {0.5em}          
  {}                

\theoremstyle{myexample}
\newtheorem{example}{Example}

\newtheoremstyle{myremark}
  {\medskipamount} 
  {0pt}            
  {\normalfont}    
  {0pt}            
  {\itshape}       
  {.}              
  {0.5em}          
  {}                

\theoremstyle{myremark}
\newtheorem*{remark}{Remark}

\theoremstyle{myremark}
\newtheorem*{remarks}{Remarks}

\newcommand{\Pp}{\mathbb P}

\newcommand{\dd}{\,\mathrm d}

\title{Polynomial Decay in Absorbing Continuous-Time \\Markov Chains}
\author{Phil Pollett}
\date{}

\begin{document}
\maketitle

\begin{center}
\begin{minipage}{14cm}
{\bf Abstract} \ \
\setlength{\parindent}{1.5em}
We study the long-time behaviour of absorbing continuous-time Markov
chains on a countable state space. We investigate when transition
probabilities between fixed states have a common polynomial correction
to their exponential decay, whether their leading coefficients have rank
one, and how these asymptotics relate to survival probabilities. We
prove that regular variation for one fixed pair of states determines the
polynomial order for every fixed pair. A ``hub-and-traps'' model is
shown to have rank-one fixed-state asymptotics, but a different
polynomial survival scale, with the surviving mass carried by
increasingly deep traps. A spectral-edge theorem gives rank-one
coefficients for birth--death processes under a regular-variation
hypothesis. In contrast, a killed random walk on a homogeneous tree has
a common fixed-state scale with a coefficient kernel of rank greater
than one. These results separate common decay, rank-one factorisation,
and survival asymptotics.
\end{minipage}
\end{center}

\smallskip
\noindent {\bf MSC 2020:} Primary 60J35; Secondary 60J27, 60K05, 47A10.

\medskip

\noindent {\bf Keywords:} absorbing continuous-time Markov chains;
polynomial asymptotics; regular variation; defective renewal theory;
spectral edge; birth--death processes; quasi-stationarity

\section{Introduction}\label{sec:introduction}

For an irreducible continuous-time Markov chain with an absorbing state,
the transition probabilities between transient states may satisfy
$p_{ij}(t)\sim a_{ij}L(t)$ as $t\to\infty$. Does the existence of a
common scale function $L$ force the coefficient matrix $A=(a_{ij})$ to
have rank one? More broadly, to what extent does a common asymptotic
scale determine the global asymptotic structure of the transition
function? These questions motivate the present paper.

Let $(X(t),\,t\geq 0)$ be a continuous-time Markov chain with an honest
transition function $P(t)=(p_{ij}(t),\,i,j\in S)$ on the countable state
space $S=\{0\}\cup C$, where $0$ is absorbing. Let $Q=(q_{ij},\,i,j\in
S)$ denote its $q$-matrix. We do not require $Q$ to be regular; thus,
when the rates do not determine the transition function uniquely, $P$ is
understood to be the particular transition function of the process under
consideration. We assume that $C$ is irreducible and that absorption
occurs with probability~$1$ from every state in~$C$. In particular, with
$T_0=\inf\{t\geq 0:X(t)=0\}$, we have
$\Pp_i(T_0>t)=1-p_{i0}(t)=\sum_{j\in C}p_{ij}(t)\to 0$ for every $i\in C$.

Of principal interest in the study of quasi-stationarity is
the limiting behaviour of the ratio
\begin{equation}
 \frac{p_{ij}(t)}{1-p_{i0}(t)} = \Pp_i(X(t)=j\mid T_0>t),
 \qquad i,j\in C,
 \label{PKP2}
\end{equation}
called the \emph{limiting conditional distribution} (LCD). 
Whether this limit exists and defines a proper distribution on~$C$
depends crucially on the relative rates at which the numerator and
denominator in~\eqref{PKP2} converge to $0$. These rates are our
principal concern here.

The first general results on LCDs~\cite{SV66} were facilitated by a finer
classification of transient states~\cite{Kingman1963} and by the
existence of a common rate at which the transition probabilities decay.
There is a $\lambda\geq 0$, called the \emph{decay parameter} of~$C$,
such that
\[
  \frac{1}{t}\log p_{ij}(t)\to -\lambda,
  \qquad t\to\infty,
\]
for all $i,j\in C$. The communicating class $C$ is called
\emph{$\lambda$-recurrent} or \emph{$\lambda$-transient} according to
whether
\[
  \int_0^\infty e^{\lambda t}p_{ij}(t)\,\mathrm dt
\]
diverges or converges, for some, and hence every, $i,j\in C$. When $C$
is $\lambda$-recurrent, it is called \emph{$\lambda$-positive} or
\emph{$\lambda$-null} according to whether
\[
  \lim_{t\to\infty}e^{\lambda t}p_{ij}(t)
\]
is positive or $0$, again for some, and hence every, $i,j\in C$.
In the $\lambda$-positive case, the limiting values are determined by
families $(m_i,\,i\in C)$ and $(x_i,\,i\in C)$ satisfying
\begin{equation}
  \sum_{i\in C} m_i p_{ij}(t) = e^{-\lambda t} m_j
  \quad \text{and} \quad
  \sum_{j\in C} p_{ij}(t) x_j = e^{-\lambda t} x_i.
\label{PKP3}
\end{equation}
When $C$ is $\lambda$-recurrent, the positive solutions to~\eqref{PKP3}, 
called the $\lambda$-invariant measure and vector,
respectively, are unique up to constant multiples.
$C$~is then $\lambda$-positive if and only if 
$A^{-1}:=\sum_{k\in C} m_k x_k <\infty$, whence we obtain the 
limit 
\[
\lim_{t\to\infty} e^{\lambda t} p_{ij}(t) = A x_i m_j.
\]
This rank-one limit suggests the following formal calculation:
\begin{equation}
 \Pp_i(X(t)=j\mid T_0>t)
 =
 \frac{e^{\lambda t}p_{ij}(t)}
      {e^{\lambda t}\sum_{k\in C}p_{ik}(t)}
 \to
 \frac{m_j}{\sum_{k\in C}m_k}.
\label{eq:formal}
\end{equation}
The calculation is formal because pointwise convergence of the
fixed-state transition probabilities does not by itself justify
interchanging limit and summation in the denominator in~\eqref{eq:formal}.
In the $\lambda$-positive case the required conclusion is nevertheless
valid~\cite{VereJones1969}, with the limit interpreted as~$0$ when
$\sum_{k\in C}m_k=\infty$.

Of course $p_{ij}(t)$ is seldom available explicitly, 
but there are ``$q$-matrix versions'' of~\eqref{PKP3},
\begin{equation}
 \sum_{i\in C} m_i q_{ij}= -\lambda  m_j
  \quad \text{and} \quad
 \sum_{j\in C} q_{ij} x_j = -\lambda x_i,
\label{PKP4}
\end{equation}
and positive solutions to~\eqref{PKP4} satisfy~\eqref{PKP3} under
conditions that are easy to check~\cite{Pollett1986}.

One might think that this completes the picture. However,
$\lambda$-positivity is not necessary for the existence of an LCD. Even
when $e^{\lambda t}p_{ij}(t)\to 0$, the transition probabilities may
satisfy $t^{\kappa}e^{\lambda t}p_{ij}(t) \to Ax_im_j$, where
$\kappa>1$, $A>0$, and where $m$ and $x$ are positive
$\lambda$-invariant measures, and positive $\lambda$-invariant vectors,
respectively. The M/M/1 queue killed at the end of its first busy period
has this property (see below). Furthermore, there might be a
$\kappa_0\leq \kappa$ such that $t^{\kappa_0} e^{\lambda t}
(1-p_{i0}(t)) \to B x_i$, where $B>0$. The killed M/M/1 queue has this
property with $\kappa_0=\kappa$. So, in principle the formal calculation
in~\eqref{eq:formal} may still yield an LCD. In contrast, the critical
Markov branching process~\cite{AsmussenHering1983} provides a further
example, but with $\kappa_0<\kappa$, and the formal calculation would
lead to a zero limit.

Although motivated by such calculations, we concentrate on the
asymptotic behaviour of fixed-state probabilities $p_{ij}(t)$ and
survival probabilities $\Pp_i(T_0>t)$ separately, because limiting
conditional behaviour is governed by the interaction between their
asymptotics. The conjectures proposed in~\cite{Pollett2022} provide the
starting point for the present investigation. We state them below in a
regular-variation form.

\begin{conjecture}[Common regular variation]
\label{conj:common}
Suppose that $C$ is $\lambda$-transient and that, for some fixed
$i_0,j_0\in C$, the residual
$e^{\lambda t}p_{i_0j_0}(t)$ is regularly varying with index
$-\kappa$ for some $\kappa>1$. Then there is a slowly varying
function $L$, common to all fixed pairs of states, such that
\[
  e^{\lambda t}p_{ij}(t)
  \sim c_{ij}t^{-\kappa}L(t),
  \qquad i,j\in C,
\]
for positive coefficients $c_{ij}$.
\end{conjecture}

\noindent
This strengthens the regular-variation formulation
in~\cite{Pollett2022} by requiring a common slowly varying factor.
The original positive-limit formulation is recovered when
$L(t)\to L_\infty\in(0,\infty)$. The restriction $\kappa>1$
reflects the original $\lambda$-transient setting. The spectral-edge
result below establishes the corresponding fixed-state asymptotics
for every $\kappa>0$, with $\kappa>1$ needed only for
$\lambda$-transience to be automatic.

\begin{conjecture}[Rank-one coefficients]
\label{conj:rankone}
The coefficients factor as
\[
  c_{ij}=Ax_im_j,
\]
where $m$ and $x$ are a positive $\lambda$-invariant measure and
vector, respectively, or, more generally, positive
$\lambda$-subinvariant quantities, with equality in~\eqref{PKP3}
replaced by the corresponding inequalities ($\leq$).
\end{conjecture}

\begin{conjecture}[Survival-scale]
\label{conj:survival}
Given the existence of $\kappa$ as in
Conjecture~\ref{conj:common}, there exist
$\kappa_0\leq\kappa$, $B>0$, and a slowly varying function $L_0$
such that
\[
  e^{\lambda t}\Pp_i(T_0>t)
  \sim Bx_it^{-\kappa_0}L_0(t).
\]
\end{conjecture}

\noindent
This formulation records the original conjectural picture set out
in~\cite{Pollett2022}.

\begin{remarks}
(i) The inclusion of slowly varying factors is important. Regular
variation naturally produces asymptotic forms of the type
$t^{-\kappa}L(t)$; it does not by itself imply convergence after
multiplication by a pure power of $t$.

\medskip
\noindent
(ii) The distinction between the decay parameter governing fixed-state
transition probabilities and that governing survival probabilities has
appeared previously in the study of conditioned processes. Jacka and
Roberts~\cite{JackaRoberts1995} introduced separate decay parameters for
these quantities and observed that they need not coincide. The analyses
below suggest that this distinction extends beyond the exponential scale
and persists in the polynomial corrections to large-$t$ asymptotics.

\medskip
\noindent
(iii) Such polynomial corrections also arise in recent work of
Champagnat and Villemonais~\cite{ChampagnatVillemonais2026} on absorbed
Markov chains with reducible state spaces. Their polynomial convergence
parameter reflects the successive passage through communicating classes
contributing at the same exponential scale. The present paper concerns
a different mechanism: decaying polynomial residuals within a single
irreducible transient class, generated by renewal tails or spectral-edge
behaviour.

\medskip
\noindent
(iv) Polynomial rates also arise in subgeometric ergodicity, including
applications to Markov chain Monte Carlo, although there the object of
study is convergence to equilibrium rather than decay in an evanescent
class; see, for example,
\cite{DoucMoulinesSoulier2006,MeynTweedie1993}.
\end{remarks}

\medskip
\noindent
\textbf{Status of the conjectures.}\quad
The results of this paper give the following picture.
Conjecture~\ref{conj:common} remains open in its stated generality.
Proposition~\ref{prop:polynomial-order-solidarity} proves the weaker
conclusion that, if the residual transition probability $e^{\lambda
t}p_{i_0j_0}(t)$ is regularly varying for one pair of states $i_0,j_0\in
C$, then, for every fixed pair $i,j\in C$, $e^{\lambda t}p_{ij}(t)$ has
the same logarithmic polynomial order and is bounded above and below by
positive constant multiples of the same regularly varying function. It
does not establish the positive ratio limits required by
Conjecture~\ref{conj:common}. Even when common fixed-state asymptotics
hold, Conjecture~\ref{conj:rankone} is not true in general. It is valid
in the scalar spectral settings treated below, including the
birth--death case, but the homogeneous-tree example has a limiting
coefficient kernel of rank greater than one.

Conjecture~\ref{conj:survival} is also not true in general. In the
homogeneous-tree example the survival probability has an exponential
decay parameter strictly smaller than the fixed-state decay parameter.
On the positive side,
Proposition~\ref{prop:tightness-survival} shows that an appropriate
asymptotic-tightness condition permits passage from fixed-state
asymptotics to survival asymptotics at the same scale.

These conclusions separate the roles of common polynomial order,
spectral multiplicity, and tightness. They also govern limiting
conditional behaviour. A proper LCD exists for the killed stable $M/M/1$
queue, but not for the hub-and-traps or homogeneous-tree examples, where
conditioned mass ``escapes'' to infinity. Thus fixed-state asymptotics
alone do not ensure a limiting conditional distribution; tightness of
the laws of $X(t)$ conditional on $T_0>t$ must also be controlled.
Section~\ref{sec:discussion} summarizes the main results in greater
detail and formulates refined conjectures suggested by the scalar and
multichannel spectral representations.

\medskip
For the analysis below, it is useful to remove the common exponential
decay rate and work directly with the residual semigroup. Define
$g_{ij}(t)=e^{\lambda t}p_{ij}(t)$ and $G(t)=e^{\lambda t}P_C(t)$, where
$P_C$ is $P$ restricted to $C$. Then $G$ retains the semigroup property,
$G(t+s)=G(t)G(s)$, although it need not be a transition function, since
$\sum_j g_{ij}(t) =e^{\lambda t}\Pp_i(T_0>t)$ may exceed $1$. This
transformation isolates the polynomial residuals that are the principal
objects of the paper and permits a range of the subsequent arguments to
be carried out with exponential decay parameter zero. Consideration of
such semigroups dates back at least to Jurkat~\cite{Jurkat1960}. Kingman
also considers them in Section~7 of~\cite{Kingman1963}, and his
Theorem~9 applies because $t^{-1}\log g_{ii}(t)$ is bounded; we conclude
that $t^{-1}\log g_{ij}(t)\to 0$. The formal generator of $G$ is
$Q_C+\lambda I$; its row sums are $\lambda-q_{i0}$ and might be
positive.

The paper studies these questions by renewal theoretic methods, scalar
and multichannel spectral representations, and an explicit
homogeneous-tree counterexample. It is organised as follows.
Section~\ref{sec:renewal} introduces the renewal framework and analyses
the killed stable $M/M/1$ queue. Section~\ref{sec:hub-traps} develops
the explicit hub-and-traps model.
Section~\ref{sec:birth-death-spectral-edge} establishes the
spectral-edge theorem for birth--death processes.
Section~\ref{sec:multichannel} discusses multichannel spectral
representations. Section~\ref{sec:tree-counterexample} presents the
homogeneous-tree counterexample. Finally,
Section~\ref{sec:discussion} summarizes the implications of the main
results for the original conjectures and formulates refined conjectures
suggested by the spectral analysis.

\section{A renewal approach}\label{sec:renewal}

The conjectures above suggest that polynomial decay may be
generated by a single exceptionally long excursion. To examine this
mechanism, fix a reference state $a\in C$ and decompose sample paths
according to successive returns to $a$. After removal of the common
exponential decay, the tail of the return-time distribution may then
determine the asymptotic behaviour of the diagonal entry $g_{aa}$. We
develop a defective-renewal equation for $g_{aa}$ and give a
subexponential sufficient condition for polynomial decay. Entrance and
exit decompositions indicate how the diagonal asymptotics can propagate
to general fixed states and lead to rank-one coefficients. The killed
stable $M/M/1$ queue illustrates the mechanism, while the hub-and-traps
model of the next section provides a setting in which it can be analysed
completely.

We first record a solidarity property that follows directly from
the semigroup property. It establishes a common polynomial order
when one fixed-state residual is regularly varying, but does not
assert the common asymptotic equivalent in Conjecture~\ref{conj:common}.

\begin{proposition}[Solidarity of polynomial order]
\label{prop:polynomial-order-solidarity}
If, for some $a,b\in C$,
\[
  g_{ab}(t)\sim t^{-\kappa}L(t),
  \qquad t\to\infty,
\]
where $L$ is slowly varying, then, for every fixed $i,j\in C$,
\[
  0<
  \liminf_{t\to\infty}
  \frac{g_{ij}(t)}{t^{-\kappa}L(t)}
  \leq
  \limsup_{t\to\infty}
  \frac{g_{ij}(t)}{t^{-\kappa}L(t)}
  <\infty.
\]
Furthermore,
\[
  \lim_{t\to\infty}
  \frac{\log g_{ij}(t)}{\log t}=-\kappa,
  \qquad i,j\in C.
\]
\end{proposition}

\begin{proof}
Fix $i,j\in C$ and $h>0$. By irreducibility,
$g_{ia}(h)$, $g_{bj}(h)$, $g_{ai}(h)$ and
$g_{jb}(h)$ are strictly positive. Using the semigroup
property, written out elementwise
(the Chapman--Kolmogorov equations), retaining just one term in each
sum, gives
\[
  g_{ij}(t+2h)
  \geq g_{ia}(h)g_{ab}(t)g_{bj}(h)
\qquad\text{and}\qquad
  g_{ab}(t+2h)
  \geq g_{ai}(h)g_{ij}(t)g_{jb}(h).
\]
Thus, for $t>2h$,
\[
  g_{ia}(h)g_{bj}(h)g_{ab}(t-2h)
  \leq g_{ij}(t)
  \leq
  \frac{g_{ab}(t+2h)}
       {g_{ai}(h)g_{jb}(h)}.
\]
Regular variation implies $g_{ab}(t\pm2h)/g_{ab}(t)\to1$. Dividing by
$t^{-\kappa}L(t)$ and using the assumed asymptotics proves the two-sided
bounds. Finally, these bounds give $\log g_{ij}(t)=-\kappa\log t+\log
L(t)+O(1)$. Since $\log L(t)/\log t\to0$, the asserted logarithmic limit
follows.
\end{proof}

This gives a partial answer to the question raised in \cite{Pollett2022}
of whether Kingman's~\cite{Kingman1963} Chapman--Kolmogorov comparisons
can be adapted to polynomial decay. We now turn to a return-time
decomposition, which yields precise asymptotics for diagonal entries
under additional assumptions.

Fix $a\in C$.  Let
$\sigma_a=\inf\{t>0:X(t)\neq a\}$,
$\tau_a^+=\inf\{t\geq\sigma_a:X(t)=a\}$,
with $\tau_a^+=\infty$ if absorption occurs first.  Let $q_a=-q_{aa}$ and define the defective return measure
$F_a(\dd s)=\Pp_a(\tau_a^+\in\dd s,\ \tau_a^+<T_0)$.
A first-return decomposition gives the exact renewal equation
\begin{equation}\label{eq:renewal-original}
 p_{aa}(t)=e^{-q_at}+\int_{(0,t]}p_{aa}(t-s)F_a(\dd s).
\end{equation}
Define the exponentially tilted return measure by
$K_a(\dd s)=F_a^{(\lambda)}(\dd s)=e^{\lambda s}F_a(\dd s)$.
Then
\begin{equation}\label{eq:renewal-tilted}
 g_{aa}(t)=e^{-(q_a-\lambda)t}
 +\int_{(0,t]}g_{aa}(t-s)K_a(\dd s).
\end{equation}
In the $\lambda$-transient case,
$\int_0^\infty g_{aa}(t)\,\mathrm dt<\infty$.
Since \eqref{eq:renewal-tilted} gives
$g_{aa}(t)\geq e^{-(q_a-\lambda)t}$,
we must have $q_a>\lambda$. We may therefore integrate
\eqref{eq:renewal-tilted}. Tonelli's theorem gives
\[
\int_0^\infty g_{aa}(t)\,\mathrm dt
=
\frac{1}{q_a-\lambda}
+
\rho_a\int_0^\infty g_{aa}(t)\,\mathrm dt,
\]
where $\rho_a=K_a((0,\infty))$.
Since the integral on the left is finite, it follows that
$\rho_a<1$. Irreducibility gives $\rho_a>0$, and hence
$0<\rho_a<1$.
Thus \eqref{eq:renewal-tilted} is a defective renewal equation, with
\[
\int_0^\infty g_{aa}(t)\,\mathrm dt
=
\frac{1}{(q_a-\lambda)(1-\rho_a)}.
\]
Its Laplace transform is
\begin{equation}\label{eq:laplace-renewal}
 \widehat g_{aa}(z)
 =\frac{1}{(z+q_a-\lambda)(1-\widehat K_a(z))}.
\end{equation}
Thus the tail of $g_{aa}$ is controlled by the singular behaviour of
$\widehat K_a$ at $0$.

\medskip
\noindent
{\bf Defective renewal with a subexponential density.}\quad
Suppose that the tilted return measure $K_a$ has a density $k_a$, and
write
$ \rho_a=\int_0^\infty k_a(s)\,\mathrm ds\in(0,1)$.
The renewal-density asymptotics needed below are a direct consequence of
standard weighted-convolution and random-sum results for
subexponential densities; see
\cite[Corollary~4, Proposition~8 and Theorem~3]
{AsmussenFossKorshunov2003} and
\cite[Corollary~4.10 and Theorems~4.11 and~4.30]
{FossKorshunovZachary2013}.
In particular, the factor $(1-\rho_a)^{-2}$ is obtained by using the
geometric coefficients $\rho_a^n$, normalized when necessary.
We state the resulting estimate in the form needed here, and include
the short argument to make the normalization and the additional
convolution with the driving term explicit.
We write $\mathcal S_d$ for the class of probability densities $f$ on
$\mathbb R_+$ that are long-tailed and satisfy
$(f\ast f)(t)\sim 2f(t)$, $t\to\infty$.

\begin{proposition}[A defective-renewal estimate]
\label{prop:defective-renewal-density}
Let $k$ be a nonnegative, locally bounded density of a defective measure on
$(0,\infty)$, with
$\rho=\int_0^\infty k(s)\,\mathrm ds\in(0,1)$.
Write $f=k/\rho$, and suppose that $f\in\mathcal S_d$.
Assume further that $f$ satisfies a local Kesten-type bound: for
every $\varepsilon\in(0,1)$, there exist constants
$C_\varepsilon<\infty$ and $t_\varepsilon<\infty$ such that
\begin{equation}
\label{eq:kesten-abstract-condition}
 f^{\ast n}(t)
 \leq C_\varepsilon(1+\varepsilon)^n f(t),
 \qquad t\geq t_\varepsilon,\quad n\geq1.
\end{equation}
Let $h$ be a nonnegative integrable function satisfying
$0<\int_0^\infty h(s)\,\mathrm ds<\infty$ and
$h(t)=o(k(t))$, $t\to\infty$, and put
$u(t):=\sum_{n\geq1}k^{\ast n}(t)$.
Suppose that convolution with $h$ preserves the tail of $u$, in the
sense that
\[
 (h\ast u)(t)
 \sim
 \left(\int_0^\infty h(s)\,\mathrm ds\right)u(t).
\]
Define
$z(t):=\sum_{n\geq 0}(h\ast k^{\ast n})(t)$, $k^{\ast0}=\delta_0$.
Then $z$ is the minimal nonnegative solution of
\[
z(t) = h(t)+\int_0^t z(t-s)k(s)\,\mathrm ds,
\]
and satisfies
\[
 z(t) \sim \frac{1}{(1-\rho)^2} \int_0^\infty h(s)\,\mathrm ds \, k(t), 
\quad t \to \infty.
\]
\end{proposition}

\begin{proof}
Iteration of the renewal equation gives
$z=h+h\ast k+h\ast k^{\ast2}+\cdots$.
More precisely, the partial sums
$z_N=\sum_{n=0}^N h\ast k^{\ast n}$, $N\geq 0$,
increase pointwise to $z$ and, for $N\geq1$, satisfy
$z_N=h+z_{N-1}\ast k$.
Monotone convergence therefore shows that $z=h+z\ast k$. If
$\widetilde z$ is any other nonnegative solution, repeated substitution
gives
$\widetilde z\geq\sum_{n=0}^N h\ast k^{\ast n}=z_N$
for every $N$. Letting $N\to\infty$ yields $\widetilde z\geq z$.
Thus $z$ is the minimal nonnegative solution.
The asymptotics of the defective renewal density is the
geometric-weight specialization of
\cite[Theorem~3]{AsmussenFossKorshunov2003}; equivalently, it follows
from \cite[Theorem~4.30]{FossKorshunovZachary2013} after normalization
of the geometric weights. The remaining step uses the stated
tail-preservation assumption for convolution with the driving term.
We give the details in order to make the normalization and the
constant explicit.

Since $k=\rho f$, we have $k^{\ast n}(t)=\rho^n f^{\ast n}(t)$.
For every fixed $n\geq1$, subexponentiality of the density gives
$f^{\ast n}(t)\sim nf(t)$, $t\to\infty$;
see \cite[Corollary~4]{AsmussenFossKorshunov2003} or
\cite[Corollary~4.10]{FossKorshunovZachary2013}. Consequently,
\[
 \frac{\rho^n f^{\ast n}(t)}{f(t)}\to n\rho^n.
\]
By the assumed Kesten-type bound \eqref{eq:kesten-abstract-condition}, 
\[
 \frac{\rho^n f^{\ast n}(t)}{f(t)} \leq C_\varepsilon\{\rho(1+\varepsilon)\}^n
\]
for every $n\geq1$ and $t\geq t_\varepsilon$.
Since $\rho<1$, choose $\varepsilon\in(0,1)$ sufficiently small that
$\rho(1+\varepsilon)<1$.
The expression on the right is summable over $n$ and is independent
of $t$. Dominated convergence therefore gives
\[
 \frac{u(t)}{f(t)}
 =
 \sum_{n\geq1}
 \frac{\rho^n f^{\ast n}(t)}{f(t)}
 \to
 \sum_{n\geq1}n\rho^n
 =
 \frac{\rho}{(1-\rho)^2}.
\]
Since $k=\rho f$, it follows that
$u(t)\sim{k(t)}/{(1-\rho)^2}$.
Since $z=h+h\ast u$, the assumed tail-preservation property gives
$(h\ast u)(t) \sim (\int_0^\infty h(s)\,\mathrm ds) u(t)$.
Since $h(t)=o(k(t))$ and
$u(t)\sim{k(t)}/{(1-\rho)^2}$, we have $h(t)=o(u(t))$. Since
$\int_0^\infty h(s)\,\mathrm ds>0$,
the tail-preservation assumption also gives
$h(t)=o((h\ast u)(t))$.
Therefore,
\[
 z(t)
 \sim
 (h\ast u)(t)
 \sim
 \frac{k(t)}{(1-\rho)^2}
 \int_0^\infty h(s)\,\mathrm ds,
\]
as required.
\end{proof}

We now apply this result to the tilted renewal equation
\eqref{eq:renewal-tilted} with $h_a(t)=e^{-(q_a-\lambda)t}$. Suppose
that $q_a>\lambda$, that the tilted return measure has a density $k_a$
with mass $\rho_a=\int_0^\infty k_a(s)\,\mathrm ds\in(0,1)$, and that
the normalized density $f_a={k_a}/{\rho_a}$ satisfies the
subexponential-density and uniform Kesten-bound hypotheses of
Proposition~\ref{prop:defective-renewal-density}. If $k_a$ is regularly
varying, then $h_a(t)=o(k_a(t))$, and convolution with the exponentially
decreasing function $h_a$ preserves the tail of the defective renewal
density. Since
\[
 \int_0^\infty h_a(s)\,\mathrm ds = \frac{1}{q_a-\lambda},
\]
Proposition~\ref{prop:defective-renewal-density} gives
\begin{equation}
\label{eq:diag-asym}
 g_{aa}(t)
 \sim
 \frac{k_a(t)}
 {(q_a-\lambda)(1-\rho_a)^2}.
\end{equation}

The probabilistic interpretation is a tilted one-large-excursion
principle. Under the original law, exceptionally long returning
excursions are exponentially rare. After exponential tilting at the
decay parameter, the return density is heavy-tailed. Among the
geometrically weighted return cycles in the defective-renewal expansion,
a total duration of order $t$ is therefore produced asymptotically by
one excursion whose duration is of order $t$. 
We next record an exact decomposition for general fixed states.
It identifies the entrance, exit and taboo terms that would need
to be controlled to extend the diagonal asymptotics.

\medskip
\noindent
{\bf Entrance and exit decompositions.}\quad
The diagonal renewal equation accounts for repeated return cycles at
the reference state $a$. For general initial and terminal states, these
cycles must be supplemented by an entrance segment leading to $a$ and
a final segment after the last visit to $a$.
For $i\neq a$, let
\[
 H_{ia}(\mathrm ds)
 =
 \Pp_i(T_a\in\mathrm ds,\,T_a<T_0),
 \qquad
 H_{ia}^{(\lambda)}(\mathrm ds)
 =
 e^{\lambda s}H_{ia}(\mathrm ds),
\]
where $T_a=\inf\{t\geq 0:X(t)=a\}$. Set $H_{aa}^{(\lambda)}=\delta_0$.
The strong Markov property gives, for $i\neq a$,
\begin{equation}
\label{eq:entrance}
 g_{ia}(t)
 =
 \int_{(0,t]}g_{aa}(t-s)H_{ia}^{(\lambda)}(\mathrm ds).
\end{equation}
Define the tilted taboo exit kernel by
\[
 k_{aj}^{(\lambda)}(u)
 =
 e^{\lambda u}
 \Pp_a\bigl(X(u)=j,\,\tau_a^+>u\bigr),
 \qquad u\geq 0.
\]
Thus $k_{aj}^{(\lambda)}(u)$ describes the final segment from $a$ to
$j$ during which no further return to $a$ occurs. Let
\[
 U_a^{(\lambda)}
 =
 \sum_{n\geq 0}K_a^{\ast n},
 \qquad
 K_a^{\ast0}=\delta_0,
\]
be the defective renewal measure generated by the tilted return
measure.

For $i\neq a$ and $j\in C$, define the tilted taboo contribution
\[
 g_{ij}^{[a]}(t)
 =
 e^{\lambda t}
 \Pp_i\bigl(X(t)=j,\,T_a>t\bigr),
\]
and set $g_{ij}^{[a]}(t)=0$ when $i=a$.
Then the strong Markov property
at the first entrance to $a$, 
followed by the renewal decomposition
into complete return cycles and a final taboo segment, gives the exact
identity
\begin{equation}
\label{eq:full-renewal}
 g_{ij}(t)
 =
 g_{ij}^{[a]}(t)
 +
 \bigl(
 H_{ia}^{(\lambda)}
 \ast U_a^{(\lambda)}
 \ast k_{aj}^{(\lambda)}
 \bigr)(t),
 \qquad i,j\in C.
\end{equation}
Here convolution between measures and functions is understood in the
usual sense. In particular,
\[
 \bigl(
 H_{ia}^{(\lambda)}
 \ast U_a^{(\lambda)}
 \ast k_{aj}^{(\lambda)}
 \bigr)(t)
 =
 \int_{[0,t]}
 \int_{[0,t-s]}
 k_{aj}^{(\lambda)}(t-s-u)\,
 U_a^{(\lambda)}(\mathrm du)\,
 H_{ia}^{(\lambda)}(\mathrm ds).
\]
With these conventions, \eqref{eq:full-renewal} also includes the
diagonal case. Indeed,
\[
 k_{aa}^{(\lambda)}(u)
 =
 e^{-(q_a-\lambda)u},
\]
and, on taking $i=j=a$, the identity reduces to the tilted renewal
representation for $g_{aa}$. Formula~\eqref{eq:full-renewal} makes the
source of a possible rank-one factorisation transparent. If the taboo
term is negligible and convolution with the entrance and exit kernels
preserves the dominant renewal tail, then the limiting coefficient
separates into an entrance factor depending on~$i$ and an exit factor
depending on~$j$. 
A renewal proof of such a factorisation would require suitable
tail-preservation conditions for the entrance and exit kernels,
together with negligibility of the taboo term.
For the killed $M/M/1$ queue, the calculation below shows that the
return kernel lies at a critical exponential boundary. This cautions
against imposing hypotheses that require exponential integrability
beyond the decay parameter. Extending the renewal argument from the
diagonal entry to all fixed pairs requires separate control of the
entrance, exit and taboo contributions in~\eqref{eq:full-renewal}.

\medskip\noindent{\bf The killed stable $M/M/1$ queue.}\quad
Let arrivals occur at rate $p$ and departures at rate $q>p$, with killing 
on first reaching $0$. Put
\[
\nu=p+q=q_1,\qquad b=\sqrt{p/q},\qquad \theta=2\sqrt{pq},
\qquad \lambda=\nu-\theta.
\]
The transition probabilities are expressible using modified Bessel
functions, and Seneta's asymptotics~\cite{Seneta1966} give
\begin{equation}\label{eq:mm1-general}
 p_{ij}(t)\sim
 \frac{2ijb^{j-i}}{\theta\sqrt{2\pi\theta}}
 e^{-\lambda t}t^{-3/2}.
\end{equation}
The solidarity proposition would give the common order
$e^{\lambda t}p_{ij}(t)\asymp t^{-3/2}$ from the diagonal
asymptotics alone. Seneta's formula gives the stronger conclusion
of an asymptotic equivalent for each fixed pair, including its
rank-one coefficient.
Hence $\kappa=3/2$ and the state-dependent coefficient factors as
$ijb^{j-i} = \bigl(ib^{1-i}\bigr)\bigl(jb^{j-1}\bigr)$.
This normalization agrees with the spectral-edge eigen-quantities
identified below in Example~\ref{ex:mm1-edge-polynomials}.
The queue is $\lambda$-transient. Nevertheless, Seneta's 
survival asymptotics yield a proper LCD~\cite{Seneta1966}.

Taking $a=1$ as the renewal state,
\[
 g_{11}(t)=e^{\lambda t}p_{11}(t) =\frac{2e^{-\theta t}}{\theta t}I_1(\theta t) \sim \frac{2}{\theta\sqrt{2\pi\theta}}t^{-3/2}.
\]
Its Laplace transform is
\[
 G(z)=\int_0^\infty e^{-zt}g_{11}(t)\dd t =\frac{2}{z+\theta+\sqrt{z(z+2\theta)}}.
\]
The tilted first-return transform is therefore
\[
 \widehat K(z) =1-\frac{1}{(z+\theta)G(z)} =\frac12\left(1- \frac{\sqrt{z(z+2\theta)}}{z+\theta}\right).
\]
In particular, $\rho=\widehat K(0)=1/2$ and
\[
 \widehat K(z) = \frac12- \frac{1}{\sqrt{2\theta}}z^{1/2} +O(z).
\]
The transform can be inverted explicitly. The modified-Bessel
transform identity gives
\[
 \int_0^\infty e^{-zt}e^{-\theta t}
 \frac{I_1(\theta t)}{t}\,\mathrm dt
 =
 \frac{z+\theta-\sqrt{z(z+2\theta)}}{\theta}.
\]
Consequently,
\[
 \widehat K(z)
 =
 \frac{\theta}{2(z+\theta)}
 \int_0^\infty e^{-zt}e^{-\theta t}
 \frac{I_1(\theta t)}{t}\,\mathrm dt,
\]
and inversion of the product yields
\[
 k(t)
 =
 \frac{\theta}{2}e^{-\theta t}
 \int_0^t\frac{I_1(\theta u)}{u}\,\mathrm du.
\]
Since
\[
 \frac{I_1(\theta u)}{u}
 \sim
 \frac{e^{\theta u}}
 {\sqrt{2\pi\theta}\,u^{3/2}},
\]
the integral is asymptotic to
\[
 \frac{e^{\theta t}}
 {\theta\sqrt{2\pi\theta}\,t^{3/2}}.
\]
Thus
\[
 k(t)\sim\frac{1}{2\sqrt{2\pi\theta}}t^{-3/2}.
\]
The hypotheses used in~\eqref{eq:diag-asym} hold here. The explicit
formula for $k$ shows that it is bounded near~$0$, while the
asymptotic result just obtained shows that it is bounded and regularly
varying with index $-3/2$ at infinity. Thus $f=k/\rho$ is a bounded
subexponential density by
Lemma~\ref{lem:rv-density-subexponential}, and it satisfies the
uniform Kesten bound by the density theorem cited in
Proposition~\ref{prop:hub-kesten-bound}. Consequently,
\[
 u(t):=\sum_{n\geq1}k^{\ast n}(t)
 \sim \frac{k(t)}{(1-\rho)^2}.
\]
Moreover, $u$ is bounded: since $k$ is bounded,
\[
 k^{\ast n}(t)\leq \|k\|_\infty\rho^{n-1},
 \qquad n\geq1.
\]
Its regularly varying tail and boundedness imply, by splitting the
convolution at $t/2$, that convolution with
$h(t)=e^{-\theta t}$ preserves its tail:
\[
 (h\ast u)(t)\sim
 \left(\int_0^\infty h(s)\,\mathrm ds\right)u(t).
\]
Also $h(t)=o(k(t))$. Hence all the hypotheses required for
\eqref{eq:diag-asym} are satisfied.
Here
$q_1-\lambda=\theta$ and $(1-\rho)^2=1/4$.
Hence \eqref{eq:diag-asym} gives
\[
g_{11}(t) \sim \frac{k(t)}{\theta/4} \sim
\frac{2}{\theta\sqrt{2\pi\theta}}t^{-3/2},
\]
in agreement with the exact diagonal asymptotics above.
This gives a probabilistic interpretation of the exponent. Under the
original law, long returning excursions from state $1$ are exponentially
rare, with a polynomial correction. After tilting at the decay
parameter, the return density has the heavy tail \[
k(t)\sim\frac{1}{2\sqrt{2\pi\theta}}t^{-3/2}. \] The corresponding
square-root singularity of the return transform is the analytic
expression of the one-large-excursion mechanism.

\section{The hub-and-traps model}
\label{sec:hub-traps}

We introduce an explicit hub-and-traps model in which polynomial decay
arises from a mixture of increasingly long holding times. The model
makes the one-long-excursion mechanism transparent and separates
fixed-state transition asymptotics from survival asymptotics. Its
transition graph is an infinite star: communication among the traps is
mediated by a distinguished hub, while their mean holding times increase
without bound.

The construction is reminiscent of the complete-graph trap models
studied by Fontes and Mathieu~\cite{FontesMathieu2008}. Their scaling
limits are described by $K$-processes with countably many stable states
and a distinguished instantaneous or fictitious state. The connection
here is only architectural. In the present model, the hub is an ordinary
stable state with positive holding time and finite outgoing rate,
killing occurs only at the hub, and the entrance and return rates are
chosen deterministically to yield explicit regularly varying transition
and survival asymptotics.

\medskip
\noindent
{\bf Definition and elementary properties.}\quad
Let
$C=\{a,1,2,\ldots\}$,
where $a$ is the hub and the positive integers are the traps. Fix
constants $c>0$, $d>0$, $\kappa>1$,
and put $r_n=cn^{-\kappa}$, $\mu_n=n^{-1}$.
At the hub, the chain is killed at rate $d$ and enters trap $n$ at
rate $r_n$. From trap $n$, it returns to the hub at rate $\mu_n$.
Thus
$q_{a0}=d$, $q_{an}=r_n$, $q_{na}=\mu_n$,
and all other off-diagonal rates are $0$. Write
\[
 R=\sum_{n\geq1}r_n=c\,\zeta(\kappa),
 \qquad
 q_a=d+R.
\]
The diagonal entries are $q_{aa}=-q_a$ and $q_{nn}=-\mu_n$, $n\geq 1$.
Since $\kappa>1$, the total entrance rate $R$ is finite. Moreover,
$\sup_{i\in C}q_i = \max\{d+R,1\} <\infty$. The $q$-matrix is therefore
stable and uniformly bounded, and thus regular, and so its minimal
process is nonexplosive. It is irreducible on~$C$.

At every visit to the hub, the probability of being killed before
entering another trap is \mbox{$d/(d+R)>0$}.
The number of completed trap excursions before absorption is therefore
geometrically distributed, and absorption occurs almost surely from
every state in $C$.
The mean holding time in trap $n$ is $1/\mu_n=n$.
Thus higher-index traps are deeper but, because $r_n=cn^{-\kappa}$,
are entered less frequently. Polynomial decay results from the balance
between their rarity and their depth.

\medskip
\noindent
{\bf Exponential-mixture asymptotics.}\quad
The long-time behaviour of the model is governed by mixtures of
exponential holding-time tails over traps of increasing depth. The
following elementary lemma gives the required asymptotics for such
mixtures.

\begin{lemma}[A regularly varying exponential mixture]
\label{lem:mixture-sum}
For $\alpha>1$, define
$S_\alpha(t) = \sum_{n\geq1}n^{-\alpha}e^{-t/n}$, $t>0$.
Then,
$S_\alpha(t) \sim \Gamma(\alpha-1)t^{1-\alpha}$, $t\to\infty$.
\end{lemma}

\begin{proof}
Define
$g_\alpha(y)=y^{-\alpha}e^{-1/y}$, $y>0$. Then
\[
 t^{\alpha-1}S_\alpha(t) = \frac{1}{t} \sum_{n\geq1} g_\alpha\left(\frac{n}{t}\right).
\]
This is a Riemann sum on $(0,\infty)$.
The function $g_\alpha$ is integrable. Indeed, after the substitution
$z=1/y$, we get
\[
 \int_0^\infty g_\alpha(y)\,\mathrm dy
 =
 \int_0^\infty y^{-\alpha}e^{-1/y}\,\mathrm dy
 =
 \int_0^\infty z^{\alpha-2}e^{-z}\,\mathrm dz
 =
 \Gamma(\alpha-1).
\]
Fix $0<\varepsilon<M<\infty$. 
Since $g_\alpha$ is continuous on $[\varepsilon,M]$, this Riemann sum
converges to the corresponding integral:
\[
 \frac{1}{t} \sum_{\varepsilon t\leq n\leq Mt}
g_\alpha\left(\frac{n}{t}\right) \to \int_\varepsilon^M g_\alpha(y)\,\mathrm dy.
\]
For the upper tail, using $e^{-1/y}\leq 1$,
\[
 \frac{1}{t}\sum_{n>Mt} g_\alpha\left(\frac{n}{t}\right) \leq t^{\alpha-1}\sum_{n>Mt}n^{-\alpha}.
\]
For $Mt>1$, comparison with the integral gives
\[
 \sum_{n>Mt}n^{-\alpha}
 \leq \int_{Mt-1}^{\infty}x^{-\alpha}\,\mathrm dx
 =\frac{(Mt-1)^{1-\alpha}}{\alpha-1}.
\]
Consequently,
\[
 \limsup_{t\to\infty}
 \frac{1}{t}\sum_{n>Mt}
 g_\alpha\left(\frac{n}{t}\right)
 \leq \frac{M^{1-\alpha}}{\alpha-1},
\]
which tends to $0$ as $M\to\infty$.

For the lower tail, differentiation gives
$g_\alpha'(y) = y^{-\alpha-2}e^{-1/y}(1-\alpha y)$.
Thus $g_\alpha$ is increasing on $(0,1/\alpha)$. Choose
$0<\varepsilon<1/\alpha$, and put
$N_t=\lceil\varepsilon t\rceil-1$.
For all sufficiently large $t$ and $1\leq n\leq N_t$, monotonicity
gives
\[
 \frac{1}{t}g_\alpha\left(\frac{n}{t}\right) \leq \int_{n/t}^{(n+1)/t}g_\alpha(y)\,\mathrm dy.
\]
Consequently,
\[
 \frac{1}{t}\sum_{1\leq n<\varepsilon t} g_\alpha\left(\frac{n}{t}\right) \leq \int_{1/t}^{\varepsilon+1/t}g_\alpha(y)\,\mathrm dy.
\]
It follows that
\[
 \limsup_{t\to\infty}
 \frac{1}{t}\sum_{1\leq n<\varepsilon t}
 g_\alpha\left(\frac{n}{t}\right)
 \leq
 \int_0^\varepsilon g_\alpha(y)\,\mathrm dy.
\]
The right-hand side tends to $0$ as $\varepsilon\downarrow0$.

Combining the middle interval with the two tail estimates gives
\[
 \frac{1}{t}\sum_{n\geq1} g_\alpha\left(\frac{n}{t}\right) \to \int_0^\infty g_\alpha(y)\,\mathrm dy = \Gamma(\alpha-1).
\]
The asserted asymptotic statement follows.
\end{proof}

Two mixtures recur throughout the section. The function $M$
corresponds to a trap visit ending near time $t$, whereas $B$
corresponds to a trap visit still being in progress at time $t$.
Lemma~\ref{lem:mixture-sum} gives their respective asymptotic orders.

\begin{corollary}
\label{cor:hub-mixtures}
Define
$M(t)=\sum_{n\geq1}r_n\mu_ne^{-\mu_nt}$
and
$B(t)=\sum_{n\geq1}r_ne^{-\mu_nt}$.
Then
\[
M(t) = c\sum_{n\geq1}n^{-(\kappa+1)}e^{-t/n} \sim c\Gamma(\kappa)t^{-\kappa},
\text{ and }
B(t) = c\sum_{n\geq1}n^{-\kappa}e^{-t/n} \sim c\Gamma(\kappa-1)t^{-(\kappa-1)}.
\]
\end{corollary}

\begin{proof}
Apply Lemma~\ref{lem:mixture-sum} first with
$\alpha=\kappa+1$ and then with $\alpha=\kappa$.
\end{proof}

We shall also need a weighted version.

\begin{lemma}[Weighted mixture limit]
\label{lem:weighted-mixture}
Let $\alpha>1$, and let $f$ be bounded and continuous on
$[0,\infty)$. Then
\[
 t^{\alpha-1} \sum_{n\geq1} n^{-\alpha}e^{-t/n}
f\left(\frac{n}{t}\right) \to
\int_0^\infty y^{-\alpha}e^{-1/y}f(y)\,\mathrm dy.
\]
\end{lemma}

\begin{proof}
Write the expression as
\[
 \frac{1}{t}\sum_{n\geq1} g_\alpha\left(\frac{n}{t}\right) f\left(\frac{n}{t}\right),
\]
where $g_\alpha(y)=y^{-\alpha}e^{-1/y}$.
On each compact interval $[\varepsilon,M]\subset(0,\infty)$,
convergence follows from the Riemann-sum theorem.
Since $f$ is bounded, the lower- and
upper-tail estimates used in the proof of
Lemma~\ref{lem:mixture-sum} remain valid after multiplication by
$\lVert f\rVert_\infty$. This proves the result.
\end{proof}

\medskip
\noindent
{\bf The defective hub-return kernel.}\quad
Let $H$ denote the initial holding time at the hub. Then
$H\sim\operatorname{Exp}(q_a)$.
If the chain leaves the hub for trap $n$, its holding time there is
independent of $H$ and has an exponential distribution with rate
$\mu_n$.
The defective density of the first return time to the hub is therefore
\begin{equation}
\label{eq:hub-return-density}
 k(t)
 =
 \sum_{n\geq1}
 \int_0^t
 r_ne^{-q_as}\mu_ne^{-\mu_n(t-s)}\,\mathrm ds.
\end{equation}
Equivalently,
\begin{equation}
\label{eq:hub-return-convolution}
 k(t)
 =
 \int_0^t e^{-q_as}M(t-s)\,\mathrm ds.
\end{equation}
Its total mass is
\[
 \rho
 =
 \int_0^\infty k(t)\,\mathrm dt
 =
 \sum_{n\geq1}
 \left(\int_0^\infty r_ne^{-q_as}\,\mathrm ds\right)
 \left(\int_0^\infty\mu_ne^{-\mu_nu}\,\mathrm du\right).
\]
Hence
\begin{equation}
\label{eq:hub-rho}
 \rho=\frac{R}{q_a}=\frac{R}{d+R}\in(0,1),
 \qquad
 1-\rho=\frac{d}{q_a}.
\end{equation}

\begin{lemma}[Tail of the hub-return density]
\label{lem:hub-return-tail}
The defective return density satisfies
\[
 k(t) \sim \frac{c\Gamma(\kappa)}{q_a}t^{-\kappa}.
\]
\end{lemma}

\begin{proof}
By Corollary~\ref{cor:hub-mixtures},
$M(t)\sim c\Gamma(\kappa)t^{-\kappa}$.
In particular, $M$ is regularly varying with index~$-\kappa$.
Fix $A>0$ and divide \eqref{eq:hub-return-convolution} into
\[
k(t) = \int_0^A e^{-q_as}M(t-s)\,\mathrm ds 
+ \int_A^t e^{-q_as}M(t-s)\,\mathrm ds.
\]

For fixed $A$, the uniform convergence theorem for regularly varying
functions~\cite[Theorem~1.2.1, p.~6]{BinghamGoldieTeugels1987},
applied with $(t-s)/t\to1$ uniformly for $0\leq s\leq A$, gives
\[
 \sup_{0\leq s\leq A} \left| \frac{M(t-s)}{M(t)}-1 \right| \to 0.
\]
Consequently,
\[
\int_0^A e^{-q_as}M(t-s)\,\mathrm ds \sim M(t)\int_0^A e^{-q_as}\,\mathrm ds.
\]

It remains to control the integral over $[A,t]$. Split this interval
at $t/2$. 
On $[A,t/2]$, Potter bounds for the regularly varying
function $M$~\cite[Theorem~1.5.6, p.~25]{BinghamGoldieTeugels1987}
give a constant $C$ such that, for all sufficiently large~$t$,
$M(t-s)\leq C M(t)$, $A\leq s\leq t/2$.
Therefore,
\[
\int_A^{t/2}e^{-q_as}M(t-s)\,\mathrm ds 
\leq C M(t)\int_A^\infty e^{-q_as}\,\mathrm ds.
\]
On $[t/2,t]$, use the bound
\[
 M(u) = \sum_{n\geq1}r_n\mu_ne^{-\mu_nu} \leq \sum_{n\geq1}r_n\mu_n
 = c\,\zeta(\kappa+1).
\]
It follows that
\[
\int_{t/2}^t e^{-q_as}M(t-s)\,\mathrm ds \leq c\,\zeta(\kappa+1)t e^{-q_at/2},
\]
which is $o(M(t))$ because $M(t)$ decays polynomially.
Letting first $t\to\infty$ and then $A\to\infty$ gives
\[
 k(t) \sim M(t)\int_0^\infty e^{-q_as}\,\mathrm ds = \frac{M(t)}{q_a}.
\]
The result follows from the asymptotic formula for $M$.
\end{proof}

\medskip
\noindent
{\bf The hub transition probability.}\quad
A first-return decomposition at the hub gives the exact defective
renewal equation
\begin{equation}
\label{eq:hub-renewal-equation}
 p_{aa}(t)
 =
 e^{-q_at}
 +
 \int_0^t p_{aa}(t-s)k(s)\,\mathrm ds.
\end{equation}
Let
$u(t)=\sum_{n\geq1}k^{\ast n}(t)$
be the defective renewal density. Iteration of
\eqref{eq:hub-renewal-equation} gives
\begin{equation}
\label{eq:hub-renewal-solution}
 p_{aa}(t)
 =
 e^{-q_at}
 +
 \int_0^t e^{-q_as}u(t-s)\,\mathrm ds.
\end{equation}
By Lemma~\ref{lem:hub-return-tail}, the defective return density $k$
is regularly varying with index $-\kappa<-1$. Let
$f(t)={k(t)}/{\rho}$.
We first record the regular-variation criterion for density
subexponentiality that will be used to verify the hypotheses of the
defective-renewal result.

\begin{lemma}[Regular variation implies density subexponentiality]
\label{lem:rv-density-subexponential}
Let $f$ be a probability density on $\mathbb R_+$, and suppose that
$f$ is eventually positive and regularly varying with index
$-\kappa$, where $\kappa>1$. Then $f$ is a subexponential
density on $\mathbb R_+$; that is,
${f(t+s)}/{f(t)}\to 1$ as $t\to\infty$,
for every fixed $s\in\mathbb R$, and
$(f\ast f)(t)\sim 2f(t)$ as $t\to\infty$.
\end{lemma}

\begin{proof}
Since $f$ is regularly varying with index $-\kappa$, there is a
slowly varying function $L$ such that
$f(t)=t^{-\kappa}L(t)$
for all sufficiently large $t$. Hence, for every fixed
$s\in\mathbb R$,
\[
 \frac{f(t+s)}{f(t)} = \left(1+\frac{s}{t}\right)^{-\kappa}
\frac{L(t+s)}{L(t)} \to 1.
\]
Thus $f$ is long-tailed.

It remains to prove the asymptotics of the convolution. Fix $A>0$. For
$t>2A$, divide the convolution into
$(f\ast f)(t) = I_1(t,A)+I_2(t,A)+I_3(t,A)$,
where
$I_1(t,A)=\int_0^A f(s)f(t-s)\,\mathrm ds$,
$I_2(t,A)=\int_A^{t-A}f(s)f(t-s)\,\mathrm ds$,
and
$I_3(t,A)=\int_{t-A}^t f(s)f(t-s)\,\mathrm ds$.
By the uniform convergence theorem for regularly varying functions,
\[
 \sup_{0\leq s\leq A} \left| \frac{f(t-s)}{f(t)}-1 \right| \to 0.
\]
Consequently,
$I_1(t,A) \sim f(t)\int_0^A f(s)\,\mathrm ds$.
After the substitution $u=t-s$, symmetry gives
$I_3(t,A)=I_1(t,A)$,
and therefore
$I_3(t,A) \sim f(t)\int_0^A f(s)\,\mathrm ds$.

We now control the middle integral. By symmetry,
$I_2(t,A) = 2\int_A^{t/2}f(s)f(t-s)\,\mathrm ds$,
where the value at the single point $t/2$ is immaterial. For
$A\leq s\leq t/2$, the ratio $(t-s)/t$ lies in $[1/2,1]$.
The uniform convergence theorem for regularly varying functions,
applied on this compact ratio interval, gives a constant
$C<\infty$, independent of $A$, such that
$f(t-s)\leq C f(t)$, $A\leq s\leq t/2$,
for all sufficiently large $t$. It follows that
$I_2(t,A) \leq 2C f(t)\int_A^{t/2}f(s)\,\mathrm ds \leq 2C f(t)\int_A^\infty f(s)\,\mathrm ds$.

Combining these estimates gives
\[
 2\int_0^A f(s)\,\mathrm ds
 \leq
 \liminf_{t\to\infty}
 \frac{(f\ast f)(t)}{f(t)}
\]
and
\[
 \limsup_{t\to\infty}
 \frac{(f\ast f)(t)}{f(t)}
 \leq
 2\int_0^A f(s)\,\mathrm ds
 +
 2C\int_A^\infty f(s)\,\mathrm ds.
\]
Because $f$ is a probability density, letting $A\to\infty$ yields
${(f\ast f)(t)}/{f(t)} \to 2$.
Thus $f$ is a subexponential density on $\mathbb R_+$.
\end{proof}

\begin{proposition}[Subexponentiality and a
Kesten bound for the return density]
\label{prop:hub-kesten-bound}
Let $k$ be the defective hub-return density and put
$f={k}/{\rho}$.
Then $f$ is a bounded subexponential probability density on $\mathbb R_+$.
In particular,
$f^{\ast n}(t)\sim nf(t)$, $t\to\infty$,
for every fixed $n\geq1$.
Moreover, for every $\varepsilon\in(0,1)$, there
exist constants $C_\varepsilon<\infty$ and $t_\varepsilon<\infty$
such that
$f^{\ast n}(t) \leq C_\varepsilon(1+\varepsilon)^n f(t)$, 
$t\geq t_\varepsilon$, $n\geq1$.
Thus the threshold $t_\varepsilon$ is independent of $n$.
\end{proposition}

\begin{proof}
By Lemma~\ref{lem:hub-return-tail},
\[
 k(t) \sim \frac{c\Gamma(\kappa)}{q_a}t^{-\kappa}.
\]
Since $\rho\in(0,1)$,
\[
 f(t) \sim \frac{c\Gamma(\kappa)} {\rho q_a}t^{-\kappa}.
\]
We first verify that $f$ is globally bounded. Indeed,
$M(u) \leq M(0) = c\,\zeta(\kappa+1)$,
and hence
\[
 k(t) = \int_0^t e^{-q_as}M(t-s)\,\mathrm ds \leq c\,\zeta(\kappa+1)
\int_0^t e^{-q_as}\,\mathrm ds \leq \frac{c\,\zeta(\kappa+1)}{q_a}.
\]
It follows that
\[
 \sup_{t\geq 0}f(t) \leq \frac{c\,\zeta(\kappa+1)} {\rho q_a} <\infty.
\]
Since
\[
 f(t)\sim \frac{c\Gamma(\kappa)}{\rho q_a}t^{-\kappa}, \quad \kappa>1,
\]
the probability density $f$ is regularly varying with index
$-\kappa<-1$.
Lemma~\ref{lem:rv-density-subexponential} therefore shows that $f$
is a subexponential density on $\mathbb R_+$. Consequently, for every
fixed $n\geq1$,
$f^{\ast n}(t)\sim nf(t)$, $t\to\infty$.
Since $f$ is a bounded subexponential density on $\mathbb R_+$,
the density form of Kesten's bound,
\cite[Proposition~8]{AsmussenFossKorshunov2003} or
\cite[Theorem~4.11]{FossKorshunovZachary2013}, gives, for every
$\varepsilon\in(0,1)$, constants
$C_\varepsilon<\infty$ and $t_\varepsilon<\infty$ such that
\[
 f^{\ast n}(t)
 \leq C_\varepsilon(1+\varepsilon)^n f(t),
 \qquad t\geq t_\varepsilon,\quad n\geq1.
\]
In particular, the threshold $t_\varepsilon$ is independent of $n$.
This is precisely the uniform estimate required in
Proposition~\ref{prop:defective-renewal-density}.
\end{proof}

Propositions~\ref{prop:hub-kesten-bound} and
\ref{prop:defective-renewal-density} therefore give
\begin{equation}
\label{eq:defective-renewal-density}
 u(t)
 :=
 \sum_{n\geq1}k^{\ast n}(t)
 \sim
 \frac{k(t)}{(1-\rho)^2}.
\end{equation}

\begin{theorem}[Hub transition asymptotics]
\label{thm:hub-diagonal}
As $t\to\infty$,
\[
 p_{aa}(t) \sim \frac{c\Gamma(\kappa)}{d^2}t^{-\kappa}.
\]
\end{theorem}

\begin{proof}
The first term in \eqref{eq:hub-renewal-solution} is exponentially
small. 
Moreover, $u$ is bounded: since $k$ is bounded and $\int_0^\infty
k(s)\,\mathrm ds=\rho<1$,
\[
k^{\ast n}(t)\leq \|k\|_\infty\rho^{n-1},
\qquad n\geq1,
\] 
and hence $u(t)\leq \|k\|_\infty/(1-\rho)$.
Since $u$ is regularly varying and $e^{-q_as}$ is integrable,
the same argument used in the proof of Lemma~\ref{lem:hub-return-tail} gives
\[
 \int_0^t e^{-q_as}u(t-s)\,\mathrm ds \sim \frac{u(t)}{q_a}.
\]
Using \eqref{eq:defective-renewal-density},
Lemma~\ref{lem:hub-return-tail}, and \eqref{eq:hub-rho},
\[
 p_{aa}(t) \sim \frac{1}{q_a} \frac{k(t)}{(1-\rho)^2} \sim \frac{1}{q_a} \frac{c\Gamma(\kappa)}{q_a} \frac{q_a^2}{d^2}t^{-\kappa}.
\]
This reduces to
\[
 p_{aa}(t) \sim \frac{c\Gamma(\kappa)}{d^2}t^{-\kappa}.
\qedhere
\]
\end{proof}

\medskip
\noindent
{\bf Fixed-state transition probabilities.}\quad
For a fixed trap $j$, the forward equations give the exact identity
\begin{equation}
\label{eq:hub-to-trap}
 p_{aj}(t)
 =
 r_j\int_0^t
 p_{aa}(s)e^{-\mu_j(t-s)}\,\mathrm ds.
\end{equation}
Equivalently, after the substitution $u=t-s$,
\[
p_{aj}(t) = r_j\int_0^t p_{aa}(t-u)e^{-\mu_ju}\,\mathrm du.
\]
Since $p_{aa}$ is regularly varying and the exponential kernel is
integrable, Theorem~\ref{thm:hub-diagonal} gives
\begin{equation}
\label{eq:hub-to-trap-asymptotic}
 p_{aj}(t)
 \sim
 \frac{r_j}{\mu_j}p_{aa}(t).
\end{equation}
If the process starts in trap $i$, it either remains there throughout
$[0,t]$, or returns to the hub at some time $s<t$. Hence
\begin{equation}
\label{eq:trap-to-state}
 p_{ij}(t)
 =
 \delta_{ij} e^{-\mu_it}
 +
 \int_0^t
 \mu_i e^{-\mu_is}p_{aj}(t-s)\,\mathrm ds,
\end{equation}
and
\begin{equation}
\label{eq:trap-to-hub}
 p_{ia}(t)
 =
 \int_0^t
 \mu_i e^{-\mu_is}p_{aa}(t-s)\,\mathrm ds.
\end{equation}
Since
$\int_0^\infty\mu_i e^{-\mu_is}\,\mathrm ds=1$,
regular variation gives
$p_{ij}(t)\sim p_{aj}(t)$, $p_{ia}(t)\sim p_{aa}(t)$.
The exponentially small term in \eqref{eq:trap-to-state} is
negligible.

We have therefore proved the following.

\begin{theorem}[Rank-one fixed-state asymptotics]
\label{thm:hub-fixed-state}
For every fixed $i,j\in C$,
$p_{ij}(t) \sim A x_i m_jt^{-\kappa}$, $t\to\infty$,
where
\[
 A=\frac{c\Gamma(\kappa)}{d^2}, \quad x_i=1,\quad i\in C,
\]
and
\[
 m_a=1,
 \qquad
 m_j=\frac{r_j}{\mu_j}=cj^{1-\kappa},
 \quad j\geq1.
\]
\end{theorem}

The right vector $x$ and left vector $m$ are $0$-subinvariant. Indeed,
for the restriction $Q_C$ of the generator to $C$,
$(Q_Cx)_a=-d$, $(Q_Cx)_n=0$, $n\geq1$,
and $(mQ_C)_a=-d$, $(mQ_C)_n=0$, $n\geq1$.
Thus the only strict loss occurs at the killed hub.

Notice that $\sum_{j\geq1}m_j = c\sum_{j\geq1}j^{1-\kappa}$.
Thus $m$ is summable precisely when $\kappa>2$. Even in that case,
however, summability of the limiting left factor does not by itself
justify termwise summation of the fixed-state asymptotics. Such an
interchange would require a summable dominating family that is uniform
in $t$, and no such domination holds here. Indeed, termwise summation
would give order $t^{-\kappa}$, whereas the survival probability is of
order $t^{-(\kappa-1)}$. The missing mass is carried by traps whose
indices grow with $t$.

To determine the survival asymptotics, we shall need the following
convolution estimate, in which the heavier tail of the trap mixture
dominates the fixed-state return probability.

\begin{lemma}[Convolution with the heavier tail]
\label{lem:survival-convolution}
Let $v$ and $w$ be nonnegative, locally bounded functions on
$[0,\infty)$. 
Suppose that
$0<\int_0^\infty v(s)\,\mathrm ds<\infty$ and
$v(t)=O(t^{-\kappa})$, $t\to\infty$,
while
$w(t)\sim Ct^{-(\kappa-1)}$, $t\to\infty$,
for some $C>0$ and $\kappa>1$. Suppose also that $w$ is eventually
decreasing. Then
\[
\int_0^t v(s)w(t-s)\,\mathrm ds \sim w(t)\int_0^\infty v(s)\,\mathrm ds.
\]
\end{lemma}

\begin{proof}
Fix $A>0$ and divide the convolution:
\[
\int_0^t v(s)w(t-s)\,\mathrm ds = I_1(t,A)+I_2(t,A)+I_3(t,A),
\]
where
\[
I_1(t,A)=\int_0^A v(s)w(t-s)\,\mathrm ds, \ 
I_2(t,A)=\int_A^{t-A}v(s)w(t-s)\,\mathrm ds, \ 
I_3(t,A)=\int_{t-A}^t v(s)w(t-s)\,\mathrm ds.
\]
For fixed $A$, regular variation gives
\[
 \sup_{0\leq s\leq A} \left| \frac{w(t-s)}{w(t)}-1 \right| \to 0.
\]
Consequently,
\[
 \frac{I_1(t,A)}{w(t)}
 \to \int_0^A v(s)\,\mathrm ds.
\]

We next control the final interval. Since $w$ is locally bounded, put
\[
 C_A=\sup_{0\leq u\leq A}w(u)<\infty.
\]
Then
$I_3(t,A) \leq C_A\int_{t-A}^t v(s)\,\mathrm ds$.
Since $v(s)=O(s^{-\kappa})$,
$I_3(t,A)=O(t^{-\kappa})$.
Because
$w(t)\sim Ct^{-(\kappa-1)}$,
it follows that
$I_3(t,A)=o(w(t))$.

For the middle interval, split it further at $t/2$. On
$[A,t/2]$, the uniform convergence theorem and Potter bounds give a
constant $C_1$ such that
$w(t-s)\leq C_1w(t)$, $A\leq s\leq t/2$,
for all sufficiently large $t$. Therefore
\[
\int_A^{t/2}v(s)w(t-s)\,\mathrm ds \leq C_1w(t)\int_A^\infty v(s)\,\mathrm ds.
\]
On $[t/2,t-A]$, put $u=t-s$. Since $v(s)=O(t^{-\kappa})$ uniformly
for $s\geq t/2$,
\[
\int_{t/2}^{t-A}v(s)w(t-s)\,\mathrm ds \leq C_2t^{-\kappa}\int_A^{t/2}w(u)\,\mathrm du.
\]
If $1<\kappa<2$, Karamata's theorem gives
$\int_A^{t/2}w(u)\,\mathrm du=O(t^{2-\kappa})$,
and hence this term is
$O(t^{2-2\kappa})=o(t^{1-\kappa})=o(w(t))$.
If $\kappa=2$, the integral is $O(\log t)$, and the same conclusion
holds. If $\kappa>2$, $w$ is integrable, and the term is
$O(t^{-\kappa})=o(w(t))$.

It follows that
\[
 \limsup_{t\to\infty}
 \frac{I_2(t,A)+I_3(t,A)}{w(t)}
 \leq
 C_1\int_A^\infty v(s)\,\mathrm ds.
\]
Letting $A\to\infty$ and using the integrability of $v$ gives
\[
\int_0^t v(s)w(t-s)\,\mathrm ds \sim w(t)\int_0^\infty v(s)\,\mathrm ds.
\qedhere
\]
\end{proof}

\medskip
\noindent
{\bf Survival asymptotics.}\quad
Let
$S_i(t)=\mathbb P_i(T_0>t) =\sum_{j\in C}p_{ij}(t)$.
Summing the exact entrance representation over the traps gives
\[
\sum_{j\geq1}p_{ij}(t) = \mathbf 1_{\{i\geq1\}}e^{-\mu_it} + \int_0^t p_{ia}(s)B(t-s)\,\mathrm ds,
\]
where $B(t)=\sum_{n\geq1}r_ne^{-\mu_nt}$.
Therefore
\begin{equation}
\label{eq:survival-decomposition}
 S_i(t)
 =
 p_{ia}(t)
 +
 \mathbf 1_{\{i\geq1\}}e^{-\mu_it}
 +
 \int_0^t p_{ia}(s)B(t-s)\,\mathrm ds.
\end{equation}

Since killing occurs at rate $d$ while the chain is at the hub, the
density of the absorption time is
$d\,p_{ia}(t)$.
Absorption occurs almost surely, so
\begin{equation}
\label{eq:hub-occupation}
 d\int_0^\infty p_{ia}(s)\,\mathrm ds=1.
\end{equation}
Thus $\int_0^\infty p_{ia}(s)\,\mathrm ds=1/d$. By
Theorem~\ref{thm:hub-fixed-state}, $p_{ia}(t)=O(t^{-\kappa})$. Moreover,
$p_{ia}$ is integrable by~\eqref{eq:hub-occupation}. The function $B$ is
decreasing, because it is a positive mixture of decreasing exponential
functions, and Corollary~\ref{cor:hub-mixtures} gives $B(t)\sim
c\Gamma(\kappa-1)t^{-(\kappa-1)}$. Lemma~\ref{lem:survival-convolution},
applied with $v=p_{ia}$, $w=B$, therefore gives
\[
\int_0^t p_{ia}(s)B(t-s)\,\mathrm ds \sim B(t)\int_0^\infty p_{ia}(s)\,\mathrm ds.
\]
Hence,
\[
 \int_0^t p_{ia}(s)B(t-s)\,\mathrm ds \sim \frac{B(t)}d.
\]
The other two terms in \eqref{eq:survival-decomposition} are
negligible relative to $B(t)$: the initial holding term is exponentially
small, while
$p_{ia}(t)=O(t^{-\kappa}) =o(t^{-(\kappa-1)})$.

We have proved the following.

\begin{theorem}[Survival asymptotics]
\label{thm:hub-survival}
For every fixed $i\in C$,
\[
 \mathbb P_i(T_0>t) \sim \frac{c\Gamma(\kappa-1)}{d} t^{-(\kappa-1)}.
\]
Consequently,
$\kappa_0=\kappa-1<\kappa$.
\end{theorem}

Theorem~\ref{thm:hub-survival} is valid for every $\kappa>1$.
The restriction $1<\kappa<2$ is therefore not needed for the survival
asymptotics themselves.

The relation $\kappa_0=\kappa-1$ has a direct probabilistic
interpretation.
A fixed-state transition probability requires a long trap visit to end
near time $t$, producing the density mixture
$\sum_n r_n\mu_n e^{-\mu_n t}$. Survival requires only that the long
trap visit remain in progress at time $t$, producing the tail mixture
$\sum_n r_n e^{-\mu_n t}$.

\medskip
\noindent
{\bf A scaled conditional limit law.}\quad
To describe the escape of conditional mass, define a rescaled variable
$Z_t$ by
\[
 Z_t=
 \begin{cases}
  X(t)/t, & X(t)\in\{1,2,\ldots\},\\
  0,      & X(t)=a.
 \end{cases}
\]
The value assigned at the hub is immaterial, because the conditional
probability of being at the hub tends to $0$.

\begin{theorem}[Scaled conditional limit]
\label{thm:hub-conditional-limit}
For every fixed initial state $i\in C$,
\[
\mathcal L_i(Z_t\mid T_0>t) \Rightarrow \mathcal L(Y),
\]
where $Y$ has density
\[
 f_Y(y) = \frac{y^{-\kappa}e^{-1/y}} {\Gamma(\kappa-1)}, \quad y>0,
\]
that is $Y^{-1}\sim\operatorname{Gamma}(\kappa-1,1)$.
\end{theorem}

\begin{proof}
Regard $Z_t$ as a random variable taking values in $[0,\infty)$, and
let $f\in C_b([0,\infty))$. The contribution from the hub is
$f(0)p_{ia}(t)=O(t^{-\kappa}) =o\bigl(t^{-(\kappa-1)}\bigr)$.
If the initial state is a trap $i\geq1$, there is also a contribution
$e^{-\mu_i t}f(i/t)$
from the event that the initial holding time has not ended by time
$t$. Since $i$ and $\mu_i$ are fixed, this contribution is
exponentially small and hence is also
$o\bigl(t^{-(\kappa-1)}\bigr)$.
For the contribution from the traps, the exact entrance representation
gives
\[
 \mathbb E_i\left[
 f(Z_t);\,X(t)\in\{1,2,\ldots\}
 \right]
 =
 \int_0^t p_{ia}(s)
 \sum_{n\geq1}
 r_ne^{-\mu_n(t-s)}
 f\left(\frac{n}{t}\right)
 \,\mathrm ds
 +
 o\left(t^{-(\kappa-1)}\right).
\]
Put
\[
 W_t(s)
 =
 t^{\kappa-1}
 \sum_{n\geq1}
 r_ne^{-\mu_n(t-s)}
 f\left(\frac{n}{t}\right),
 \qquad 0\leq s\leq t.
\]
For every fixed $A>0$, the weighted Riemann-sum argument used in
the proof of
Lemma~\ref{lem:weighted-mixture} is uniform for $0\leq s\leq A$.
Indeed,
\[
  e^{-\mu_n(t-s)}
  =
  \exp\left\{-\frac{1-s/t}{n/t}\right\},
\]
and $1-s/t\to1$ uniformly for $0\leq s\leq A$. On every compact
interval $[\varepsilon,M]\subset(0,\infty)$, the corresponding
integrands therefore converge uniformly in both $s$ and the
Riemann-sum variable. Moreover, for all sufficiently large $t$,
$\frac12\leq 1-{s}/{t}\leq 1$, $0\leq s\leq A$.
The upper-tail estimate from Lemma~\ref{lem:weighted-mixture} is
therefore uniform in $s$, while the lower tail is bounded uniformly
by a constant multiple of
$y^{-\kappa}e^{-1/(2y)}$.
Since this function is integrable near $0$, the lower-tail estimate
is uniform as well. Consequently,
\[
\sup_{0\leq s\leq A} \left| W_t(s) - c\int_0^\infty
y^{-\kappa}e^{-1/y}f(y)\,\mathrm dy \right| \to 0.
\]
It follows that
\[
 \int_0^A p_{ia}(s)W_t(s)\,\mathrm ds
 \to
 c\left(\int_0^A p_{ia}(s)\,\mathrm ds\right)
 \int_0^\infty
 y^{-\kappa}e^{-1/y}f(y)\,\mathrm dy.
\]
For $0\leq s\leq t/2$, put $u=t-s$, so that
$t/2\leq u\leq t$. By boundedness of $f$,
$|W_t(s)| \leq \|f\|_\infty t^{\kappa-1}B(u)$.
Since
$B(u)\sim c\Gamma(\kappa-1)u^{-(\kappa-1)}$
and $t/u\leq2$ on this interval, there is a constant $C_f<\infty$
such that
\[
  \sup_{0\leq s\leq t/2}|W_t(s)|\leq C_f
\]
for all sufficiently large $t$. Hence
\[
\left| \int_A^{t/2}p_{ia}(s)W_t(s)\,\mathrm ds \right| \leq C_f\int_A^\infty p_{ia}(s)\,\mathrm ds,
\]
which can be made arbitrarily small by taking $A$ large.

It remains to consider $t/2<s<t$. By
Theorem~\ref{thm:hub-fixed-state}, $p_{ia}$ is regularly varying with
index $-\kappa$. Hence the uniform convergence theorem, or Potter
bounds, gives
\[
  \sup_{t/2\leq s\leq t}p_{ia}(s)=O(t^{-\kappa}).
\]
Using
$\sum_{n\geq1}r_ne^{-\mu_n(t-s)}=B(t-s)$
and then making the substitution $u=t-s$, we obtain
\[
t^{\kappa-1}
 \int_{t/2}^t p_{ia}(s)
 \sum_{n\geq1}r_ne^{-\mu_n(t-s)}
 \left|f\left(\frac{n}{t}\right)\right|
 \,\mathrm ds
\leq
 C\|f\|_\infty t^{-1}
 \int_0^{t/2}B(u)\,\mathrm du.
\]
Since
$B(u)\sim c\Gamma(\kappa-1)u^{-(\kappa-1)}$
as $u\to\infty$, the right-hand side tends to $0$ for every
$\kappa>1$.
Indeed, it is respectively of order
\[
 t^{1-\kappa},\quad \frac{\log t}{t},\quad \frac1t,
\]
according to whether $1<\kappa<2$, $\kappa=2$, or $\kappa>2$.
Combining the three time intervals, then letting first $t\to\infty$
and afterwards $A\to\infty$, and using
$\int_0^\infty p_{ia}(s)\,\mathrm ds=1/d$, we obtain
\[
 t^{\kappa-1} \mathbb E_i\left[ f(Z_t);\,T_0>t \right] \to \frac{c}{d} \int_0^\infty y^{-\kappa}e^{-1/y}f(y)\,\mathrm dy.
\]
By Theorem~\ref{thm:hub-survival},
\[
 t^{\kappa-1}\mathbb P_i(T_0>t) \to \frac{c\Gamma(\kappa-1)}d.
\]
Dividing the two limits yields
\[
 \mathbb E_i[f(Z_t)\mid T_0>t] \to \frac{1}{\Gamma(\kappa-1)} \int_0^\infty y^{-\kappa}e^{-1/y}f(y)\,\mathrm dy.
\]
Since convergence holds for every
$f\in C_b([0,\infty))$, the asserted weak convergence is proved.
\end{proof}

\medskip
\noindent
{\bf Interpretation and relation to trap models.}\quad
The fixed-state and survival asymptotics differ because they probe
different parts of the trap mixture. For every fixed $i,j\in C$,
$p_{ij}(t)$ is of order $t^{-\kappa}$, whereas survival is of order
$t^{-(\kappa-1)}$ and is carried by traps whose indices are comparable
to $t$. Consequently, the fixed-state limits cannot be summed to obtain
the survival asymptotics.
Indeed, for every fixed $i,j\in C$,
\[
 \Pp_i(X(t)=j\mid T_0>t)
 =
 \frac{p_{ij}(t)}{\Pp_i(T_0>t)}
 \to 0.
\]
Thus there is no proper LCD on the original countable state space.
The scaled conditional limit makes this failure of tightness explicit.
Given survival up to time $t$, the process is typically in a trap with
index comparable to $t$, and $X(t)/t$ has a nondegenerate limiting
distribution.

\section{Karlin--McGregor theory and the spectral edge}
\label{sec:birth-death-spectral-edge}

The hub-and-traps model provides an explicit example in which the
renewal mechanism can be analysed completely and leads to common
polynomial decay with rank-one coefficients. We now turn to a different
approach based on spectral representations. For birth--death processes,
the Karlin--McGregor formula reduces the long-time problem to the
behaviour of a single spectral measure near its lower edge, making it
possible to derive common exponents and rank-one coefficients directly
from spectral-edge asymptotics.

In this section only, we allow a broader class than that covered by the
standing assumptions in the Introduction. The birth and death rates need
not determine a unique transition function, and absorption at the
boundary need not occur with probability $1$. We consider a transition
function satisfying both the backward and forward equations
and having the corresponding Karlin--McGregor representation described
below.

Consider a birth--death process on
$C=\{1,2,\ldots\}$,
killed if it reaches $0$.  Write
\[
 b_i=q_{i,i+1}>0,
 \qquad
 d_i=q_{i,i-1}>0,
 \qquad i\geq1,
\]
where $d_1=q_{10}$ is the killing rate at the boundary. Define the
usual potential coefficients by
\[
 \pi_1=1, \quad \pi_i= \frac{b_1b_2\cdots b_{i-1}} {d_2d_3\cdots d_i}, \quad i\geq2.
\]
Let $Q_1(x)=1$, and define the birth--death polynomials recursively by
\[
b_1Q_2(x)=b_1+d_1-x 
\quad\text{and}\quad
b_iQ_{i+1}(x) = (b_i+d_i-x)Q_i(x)-d_iQ_{i-1}(x), \quad i\geq2.
\]
For a transition function satisfying both the backward and forward
equations, the Karlin--McGregor representation takes the
form
\begin{equation}
\label{eq:KM}
 p_{ij}(t)
 =
 \pi_j\int_{[0,\infty)}
 e^{-xt}Q_i(x)Q_j(x)\,\psi(\mathrm dx),
 \qquad i,j\in C,
\end{equation}
where $\psi$ is the probability spectral measure corresponding to the
transition function under consideration. The birth--death recurrence
need not determine a unique solution of the associated Stieltjes moment
problem. In the indeterminate case, different spectral measures
correspond to different boundary behaviour at infinity and hence to
different transition functions satisfying the backward and forward
equations; see \cite[Sections~2, 3.1 and 4.1]{KijimaEtAl1997} and
\cite[Section~8.3]{Anderson1991}. For the Karlin--McGregor
representation and our normalization, see
\cite[Section~2, especially equations~(2.1), (2.2), (2.7)--(2.9)]
{KijimaEtAl1997},
\cite[Section~2, especially equations~(2.2), (2.8) and~(2.9)]
{VanDoorn1991}, and
\cite[Theorem~8.2.1]{Anderson1991}.
This is the index-shifted version of the normalization used in
\cite{KijimaEtAl1997,VanDoorn1991}: their transient states are
$\{0,1,\ldots\}$, and the absorbing state associated with positive
boundary killing is labelled $-1$.

Set $\lambda=\inf\operatorname{supp}(\psi)$.
The ``zero-interlacing theory'' also gives positivity of the birth--death
polynomials at the lower spectral endpoint; see
\cite[Sections~8.2--8.3]{Anderson1991} and
\cite[Section~2, especially equations~(2.11)--(2.15)]{VanDoorn1991}.
For $i\geq2$, let $z_{i,1}$ denote the smallest zero of $Q_i$, and let
$\xi_1$ be the limit of the smallest zeros as the polynomial degree
tends to infinity. Strict interlacing gives
$\xi_1<z_{i,1}$, $i\geq2$.
The natural orthogonality measure has lower support endpoint $\xi_1$.
In the indeterminate case, the natural measure is the orthogonality
measure for which the infimum of the support is maximal; see
\cite[Section~2]{VanDoorn2015}. Consequently, the spectral measure
$\psi$ in~\eqref{eq:KM} satisfies
\[
  0\leq\lambda
  =
  \inf\operatorname{supp}(\psi)
  \leq\xi_1
  <z_{i,1},
  \qquad i\geq2.
\]
Moreover, the boundary normalization gives
\[
  Q_i(0)
  =
  1+d_1\sum_{k=1}^{i-1}\frac{1}{b_k\pi_k}
  >0,
  \qquad i\geq2.
\]
Since $Q_i$ has no zero in $[0,z_{i,1})$, it follows that
$Q_i(\lambda)>0$, $i\geq2$.
The same conclusion is immediate for $i=1$, because $Q_1\equiv1$.
Thus $Q_i(\lambda)>0$, $i\in C$.

In~\eqref{eq:KM}, all transition probabilities are represented in terms
of a single scalar spectral measure~$\psi$, with the state dependence
entering through the product $Q_i(x)Q_j(x)$. This structure suggests
that regular variation of $\psi$ at its lower endpoint should produce a
common fixed-state exponent and rank-one coefficients. In this setting
the asymptotics are governed by a single scalar spectral measure, and
the next theorem shows that the resulting coefficient kernel is
necessarily rank one. The positivity established above excludes
cancellation of the leading edge contribution and ensures that the
resulting left and right factors are strictly positive.

\begin{theorem}[Spectral-edge theorem for birth--death processes]
\label{thm:bd-spectral-edge}
Let $p_{ij}(t)$ be the transition function of an irreducible
birth--death process on $C=\{1,2,\ldots\}$, killed if it reaches~$0$.
Suppose that $p_{ij}(t)$ satisfies both the backward and forward
equations, and that its rates, potential coefficients,
polynomials and spectral measure are normalized as above. In
particular, assume that
\begin{equation}
\label{eq:bd-KM-edge}
 p_{ij}(t)
 =
 \pi_j\int_{[\lambda,\infty)}
 e^{-xt}Q_i(x)Q_j(x)\,\psi(\mathrm dx),
 \qquad i,j\in C,
\end{equation}
where
$\lambda=\inf\operatorname{supp}(\psi)$.
Suppose that, for some $c>0$, $\kappa>0$, and a slowly varying
function $L$,
\begin{equation}
\label{eq:bd-edge-law}
 \psi([\lambda,\lambda+y])
 \sim
 c y^\kappa L(1/y),
 \qquad y\downarrow0.
\end{equation}
Then, for every fixed $i,j\in C$,
\begin{equation}
\label{eq:bd-transition-asymptotic}
 e^{\lambda t}p_{ij}(t)
 \sim
 c\Gamma(\kappa+1)
 Q_i(\lambda)\pi_jQ_j(\lambda)
 t^{-\kappa}L(t),
 \qquad t\to\infty.
\end{equation}
In particular, $\lambda$ is the common decay parameter,
and the limiting coefficients have the rank-one factorisation
$c_{ij}=A x_i m_j$,
where
$A=c\Gamma(\kappa+1)$, $x_i=Q_i(\lambda)$, $m_j=\pi_jQ_j(\lambda)$.
If, in addition,
$\int_1^\infty t^{-\kappa}L(t)\,\mathrm dt<\infty$,
then the class $C$ is $\lambda$-transient for the transition function 
under consideration.
This integrability condition holds whenever $\kappa>1$.
\end{theorem}

\begin{remark}[Checking the spectral-edge hypothesis]
Condition~\eqref{eq:bd-edge-law} holds, in particular, if $\psi$ has a
density $w$ satisfying
\[
  w(\lambda+y)
  \sim c\kappa y^{\kappa-1}L(1/y),
  \qquad y\downarrow0.
\]
Thus, when the Karlin--McGregor measure is explicit, the condition can
usually be checked by expanding its density at the lower spectral
endpoint.
\end{remark}

\begin{proof}
Define the shifted spectral measure $\widetilde\psi$ on
$[0,\infty)$ by
$\widetilde\psi(B)=\psi(\lambda+B)$.
Then \eqref{eq:bd-edge-law} becomes
\begin{equation}
\label{eq:bd-shifted-edge-law}
 \widetilde\psi([0,y])
 \sim
 c y^\kappa L(1/y),
 \qquad y\downarrow0.
\end{equation}
In particular, since $\kappa>0$,
$\widetilde\psi(\{0\})=0;$
there is no atomic contribution at the lower spectral endpoint.
Multiplying \eqref{eq:bd-KM-edge} by $e^{\lambda t}$ gives
\begin{equation}
\label{eq:bd-shifted-integral}
 e^{\lambda t}p_{ij}(t)
 =
 \pi_j\int_{[0,\infty)}
 e^{-ty}R_{ij}(y)\,\widetilde\psi(\mathrm dy),
\end{equation}
where
$R_{ij}(y)=Q_i(\lambda+y)Q_j(\lambda+y)$.
Since $Q_i$ and $Q_j$ are polynomials,
$R_{ij}(y)\to R_{ij}(0) =Q_i(\lambda)Q_j(\lambda)>0$ $(y\downarrow0)$.
Thus, by continuity, $R_{ij}$ is positive on some interval
$[0,\eta]$.

Let
$U(y)=\widetilde\psi([0,y])$, $y\geq 0$.
The function $U$ is non-decreasing and right-continuous. With
$c_0=c\Gamma(\kappa+1)$,
equation \eqref{eq:bd-shifted-edge-law} may be written as
\[
 U(y) \sim \frac{c_0}{\Gamma(\kappa+1)} y^\kappa L(1/y), 
\quad y\downarrow0.
\]
The Abelian part of Karamata's theorem for Laplace--Stieltjes transforms
at the origin
\cite[Theorem~1.7.1\ensuremath{^{\prime}}, pp.~37--38]{BinghamGoldieTeugels1987}
therefore gives
\begin{equation}
\label{eq:bd-karamata}
 \int_{[0,\infty)}e^{-ty}\,\widetilde\psi(\mathrm dy)
 \sim
 c\Gamma(\kappa+1)t^{-\kappa}L(t),
 \qquad t\to\infty.
\end{equation}
This explicit normalization accounts for the factor
$\Gamma(\kappa+1)$.

Fix $\varepsilon\in(0,1)$. By continuity of $R_{ij}$, there is
$\eta_\varepsilon>0$ such that
\[
 (1-\varepsilon)R_{ij}(0)
 \leq R_{ij}(y)
 \leq (1+\varepsilon)R_{ij}(0),
 \qquad 0\leq y\leq\eta_\varepsilon.
\]
Consequently,
\[
 (1-\varepsilon)R_{ij}(0)
 \int_{[0,\eta_\varepsilon]}e^{-ty}
 \,\widetilde\psi(\mathrm dy)
 \leq
 \int_{[0,\eta_\varepsilon]}e^{-ty}R_{ij}(y)
 \,\widetilde\psi(\mathrm dy)
 \leq
 (1+\varepsilon)R_{ij}(0)
 \int_{[0,\eta_\varepsilon]}e^{-ty}
 \,\widetilde\psi(\mathrm dy).
\]
For fixed $\eta_\varepsilon>0$, the part of the integral in
\eqref{eq:bd-karamata} over $[\eta_\varepsilon,\infty)$ is
exponentially small. Hence
\begin{equation}
\label{eq:bd-local-mass}
 \int_{[0,\eta_\varepsilon]}e^{-ty}
 \,\widetilde\psi(\mathrm dy)
 \sim
 c\Gamma(\kappa+1)t^{-\kappa}L(t).
\end{equation}

It remains to control the factor $R_{ij}$ away from the edge. Fix
$t_0>0$. For $t>t_0$ and $y\geq\eta_\varepsilon$,
$e^{-ty} \leq e^{-(t-t_0)\eta_\varepsilon}e^{-t_0y}$.
Therefore,
\[
 \left|
 \int_{[\eta_\varepsilon,\infty)}
 e^{-ty}R_{ij}(y)\,\widetilde\psi(\mathrm dy)
 \right|
 \leq
 e^{-(t-t_0)\eta_\varepsilon}
 \int_{[0,\infty)}
 e^{-t_0y}|R_{ij}(y)|\,\widetilde\psi(\mathrm dy).
\]
By Cauchy--Schwarz, the last integral is at most
\[
 \left(
 \int_{[0,\infty)}
 e^{-t_0y}Q_i(\lambda+y)^2\,\widetilde\psi(\mathrm dy)
 \right)^{1/2}
 \left(
 \int_{[0,\infty)}
 e^{-t_0y}Q_j(\lambda+y)^2\,\widetilde\psi(\mathrm dy)
 \right)^{1/2}.
\]
Both factors are finite. Indeed, the diagonal Karlin--McGregor
representation gives
\[
 \int_{[0,\infty)} e^{-t_0y}Q_i(\lambda+y)^2\,\widetilde\psi(\mathrm dy) = \frac{e^{\lambda t_0}p_{ii}(t_0)}{\pi_i} <\infty,
\]
and similarly with $i$ replaced by $j$. Since an exponentially
decreasing function is negligible relative to the regularly varying
function $t^{-\kappa}L(t)$, it follows that 
\begin{equation}
\label{eq:bd-away-edge}
 \int_{[\eta_\varepsilon,\infty)}
 e^{-ty}R_{ij}(y)\,\widetilde\psi(\mathrm dy)
 =o\bigl(t^{-\kappa}L(t)\bigr).
\end{equation}
Combining the preceding bounds with
\eqref{eq:bd-local-mass} and \eqref{eq:bd-away-edge}, and then letting
$\varepsilon\downarrow0$, yields
\[
\int_{[0,\infty)} e^{-ty}R_{ij}(y)\,\widetilde\psi(\mathrm dy) \sim c\Gamma(\kappa+1) Q_i(\lambda)Q_j(\lambda)t^{-\kappa}L(t).
\]
Substitution into \eqref{eq:bd-shifted-integral} proves
\eqref{eq:bd-transition-asymptotic} and the stated rank-one
factorisation.
Since $L$ is slowly varying, ${\log L(t)}/{\log t}\to 0$
by~\cite[Proposition~1.3.6(i), p.~16] {BinghamGoldieTeugels1987}. Hence
\[
 \frac{1}{t}\log(t^{-\kappa}L(t))\to 0.
\]
Since the coefficient in \eqref{eq:bd-transition-asymptotic} is
strictly positive, that equation implies
\[
 \lim_{t\to\infty}\frac{1}{t}\log p_{ij}(t)=-\lambda.
\]
Thus $\lambda$ is the common decay parameter.

Finally, for every fixed $T<\infty$,
\[
\int_0^T e^{\lambda t}p_{ij}(t)\,\mathrm dt \leq \int_0^T e^{\lambda t}\,\mathrm dt <\infty.
\]
Under the additional integrability assumption,
\eqref{eq:bd-transition-asymptotic} implies that
$\int_T^\infty e^{\lambda t}p_{ij}(t)\,\mathrm dt<\infty$
for all sufficiently large $T$.
Hence the class $C$ is $\lambda$-transient for the transition function 
under consideration.

If $\kappa>1$, then $t^{-\kappa}L(t)$ is regularly varying with
index $-\kappa<-1$. Karamata's theorem for tail integrals
\cite[Proposition~1.5.10, p.~27]{BinghamGoldieTeugels1987}
gives
\[
\int_x^\infty t^{-\kappa}L(t)\,\mathrm dt 
\sim \frac{x^{1-\kappa}L(x)}{\kappa-1}, 
\quad x\to\infty.
\]
In particular, the required tail integral is finite.
\end{proof}

\begin{remark}
The theorem requires only $\kappa>0$ for the fixed-state asymptotics.
The stronger condition $\kappa>1$ is needed only for one to conclude
that $C$ is $\lambda$-transient. At the borderline $\kappa=1$,
integrability depends on the slowly varying factor $L$.
\end{remark}

\begin{remark}[Dependence on the spectral measure]
Theorem~\ref{thm:bd-spectral-edge} concerns the spectral measure of the
transition function under consideration and does not require the
associated Stieltjes moment problem to be determinate. Thus, in the
indeterminate case, its conclusions apply separately to each transition
function whose spectral measure satisfies~\eqref{eq:bd-edge-law}.
\end{remark}

\begin{remark}[The spectral-edge eigen-quantities]
\label{rem:bd-edge-eigenobjects}
Define $x_i=Q_i(\lambda)$, and $m_i=\pi_iQ_i(\lambda)$.
Both $x$ and $m$ are strictly positive. The birth--death recurrence
and the detailed-balance identities give, componentwise:
$Q_Cx=-\lambda x$ and $mQ_C=-\lambda m$.  Thus the factors in
$c_{ij}=c\Gamma(\kappa+1)x_i m_j$
are the right and left $\lambda$-eigen-quantities selected by the
spectral-edge normalization.

When $P_C(t)$ is the minimal $Q$-transition function,
Proposition~5.4.1 of \cite{Anderson1991} gives the corresponding
subinvariance relations
\[
  P_C(t)x\leq e^{-\lambda t}x,
  \qquad
  mP_C(t)\leq e^{-\lambda t}m.
\]
Full semigroup invariance does not follow from the generator
eigen-equations alone. For the right vector, equality
$P_C(t)x=e^{-\lambda t}x$
is equivalent to honesty of the minimal process obtained by the
corresponding Doob transform; see
\cite[Lemma~5.4.2 and Theorem~5.4.3]{Anderson1991}.
The limiting-zero argument above shows that the spectral-edge
eigenvector selected by the birth--death normalization is strictly
positive throughout $C$.
\end{remark}

\begin{example}[Spectral-edge eigen-quantities for the killed
$M/M/1$ queue]
\label{ex:mm1-edge-polynomials}
Return to the killed stable $M/M/1$ queue considered in
\eqref{eq:mm1-general}, with constant birth and death rates $b_i=p$ and
$d_i=q$, where $0<p<q$, and put $b=\sqrt{p/q}$.
Its lower spectral endpoint is
$\lambda=p+q-2\sqrt{pq}$.
At $x=\lambda$, the birth--death recurrence becomes
\[
pQ_{i+1}(\lambda) = 2\sqrt{pq}\,Q_i(\lambda)-qQ_{i-1}(\lambda), \qquad
i\geq2.
\]
The corresponding characteristic equation is
\[
pr^2-2\sqrt{pq}\,r+q=0,
\] 
which has the repeated root $r=b^{-1}$.
The boundary recurrence gives
$pQ_2(\lambda) = p+q-\lambda = 2\sqrt{pq}$,
and hence
$Q_2(\lambda)=2b^{-1}$.
Together with $Q_1(\lambda)=1$, this yields
$x_i:=Q_i(\lambda)=i b^{1-i}$, $i\geq1$.
In particular,
$Q_i(\lambda)>0$, $i\geq1$.
The potential coefficients are
$\pi_i = ({p}/{q})^{i-1} = b^{2(i-1)}$, $i\geq1$.
The corresponding left edge object is therefore
$m_i:=\pi_iQ_i(\lambda) = i b^{i-1}$.
The recurrence and detailed-balance identities give
$Q_Cx=-\lambda x$, $mQ_C=-\lambda m$.
In this example the associated Doob-transformed birth and death rates
are
\[
  p\frac{x_{i+1}}{x_i}
  =
  \sqrt{pq}\frac{i+1}{i},
  \qquad
  q\frac{x_{i-1}}{x_i}
  =
  \sqrt{pq}\frac{i-1}{i},
  \qquad i\geq2,
\]
with transformed birth rate $2\sqrt{pq}$ at state $1$. These rates are
bounded, so the transformed minimal process is honest. Consequently,
$P_C(t)x=e^{-\lambda t}x$.
By reversibility with respect to $\pi$, and since $m_i=\pi_i x_i$, it
also follows that
$mP_C(t)=e^{-\lambda t}m$. 
Thus, with the normalization used in \eqref{eq:mm1-general},
\[
x_i=ib^{1-i}, \qquad m_j=jb^{j-1}, \qquad x_i m_j=ijb^{j-i}.
\]
Hence the two factors identified in the renewal-section asymptotics are
precisely the right and left spectral-edge eigen-quantities.
\end{example}

The decisive feature of Theorem~\ref{thm:bd-spectral-edge} is the scalar
Karlin--McGregor representation. The next section examines what remains
true when several spectral components contribute at the lower edge.

\section{Multichannel spectral representations}\label{sec:multichannel}

The preceding section derives rank-one asymptotics from a scalar
Karlin--McGregor representation. We now ask what replaces that
conclusion when a reversible transition function has several spectral
components. We call each component, comprising a spectral measure and
its associated generalized eigenfunctions, a \emph{spectral channel}.
The analysis suggests that rank one is governed not by irreducibility
alone but by effective spectral multiplicity one at the dominant edge.

Let $P_C(t)$ be a transition function that is reversible with respect
to a strictly positive measure~$m$, meaning that
\[
  m_i p_{ij}(t)=m_j p_{ji}(t),
  \qquad i,j\in C,\quad t\geq 0.
\]
Equivalently, $P_C(t)$ is a symmetric operator on $\ell^2(m)$ for
every $t\geq 0$. Kendall~\cite{Kendall1959} showed that, for the minimal
$Q$-transition function, this reversibility follows from the detailed-balance
identities
\[
  m_iq_{ij}=m_jq_{ji},
  \qquad i,j\in C.
\]
Let $A$ denote the nonnegative self-adjoint generator of this symmetric
semigroup, so that $P_C(t)=e^{-tA}$. For the normalized coordinate vectors
\[
  e_i=\frac{\mathbf 1_{\{i\}}}{\sqrt{m_i}},
\]
we have
\[
  \langle e_i,P_C(t)e_j\rangle_{\ell^2(m)}
  =
  \sqrt{\frac{m_i}{m_j}}\,p_{ij}(t).
\]
If $E$ is the projection-valued spectral measure of $A$, define
\[
  \gamma_{ij}(B)
  =
  \langle e_i,E(B)e_j\rangle_{\ell^2(m)},
\qquad B\in \mathcal B([0,\infty)).
\]
Applying the spectral theorem for self-adjoint operators to
$P_C(t)=e^{-tA}$ gives
\begin{equation}
\label{eq:Kendall}
  \sqrt{\frac{m_i}{m_j}}\,p_{ij}(t)
  =
  \int_{[0,\infty)}e^{-xt}\gamma_{ij}(\dd x).
\end{equation}
The cross-spectral measures $\gamma_{ij}$ are finite signed measures,
while each diagonal measure $\gamma_{ii}$ is a probability measure.
These integral representations are associated with Kendall's
unitary-dilation and spectral approach
\cite{Kendall1958,Kendall1959}.

Anderson and McDunnough~\cite{AndersonMcDunnough1990} gave an
alternative derivation of Kendall's representation for the minimal
$Q$-transition function using finite sections and the backward and
forward integral recursions. More importantly for the present
discussion, they gave sufficient conditions under which Kendall's
cross-spectral measures admit a product resolution.

To make this precise, let $E_0\subset C$ be a ``seed set''. For each
$\alpha\in E_0$, define the generalized eigenfunctions $Q_i^\alpha(x)$
by
\[
-xQ_i^\alpha(x) = \sum_{k\in C}q_{ik}Q_k^\alpha(x), \qquad i\in C,
\]
together with the boundary normalization
$Q_\beta^\alpha(x)=\delta_{\alpha\beta}$, $\alpha,\beta\in E_0$.
Suppose that every state has finitely many neighbours and that the
whole state space can be generated from $E_0$ by successive
continuation points in the sense of
\cite[Section~4]{AndersonMcDunnough1990}. These conditions ensure that
the generalized eigenfunctions are uniquely determined and that, for
each fixed state $i$, only finitely many $Q_i^\alpha$ are nonzero.

Writing
$\psi_\alpha=\gamma_{\alpha\alpha}$, $\alpha\in E_0$,
for the diagonal Kendall spectral measures, Anderson and McDunnough's
representation is
\begin{equation}
\label{eq:AM-multi}
 p_{ij}(t)
 =
 m_j\sum_{\alpha,\beta\in E_0}
 \frac{1}{\sqrt{m_\alpha m_\beta}}
 \int_{[0,\infty)}e^{-xt}
 Q_i^\alpha(x)Q_j^\beta(x)\,
 \gamma_{\alpha\beta}(\mathrm dx).
\end{equation}
The seed states index a matrix of spectral measures
$(\gamma_{\alpha\beta})_{\alpha,\beta\in E_0}$, whose diagonal entries
are $\psi_\alpha=\gamma_{\alpha\alpha}$. The off-diagonal measures
record interactions between distinct seed states. For every fixed $i,j$,
the double sum in~\eqref{eq:AM-multi} has only finitely many nonzero
terms.
Birth--death processes correspond to a singleton seed set.
Bilateral birth--death processes, with a doubly infinite
nearest-neighbour state space, require a two-point seed set, while the
competition-process examples use larger boundary seed sets; see
\cite[Section~4, especially Proposition~4.2]
{AndersonMcDunnough1990}, together with \cite{Iglehart1964,Pruitt1963}.

\medskip\noindent{\bf Cyclicity and spectral multiplicity.}\quad
Cyclicity provides the simplest route from the general spectral theorem to a scalar representation.
Fix a reference state $a\in C$. The normalized coordinate vector
$e_a$ is called cyclic for $A$ if
\[
  \overline{\operatorname{span}}
  \{f(A)e_a:f\text{ is bounded and Borel measurable}\}
  =
  \ell^2(m).
\]
Equivalently, the smallest closed reducing subspace for $A$ containing
$e_a$ is the whole of $\ell^2(m)$.
If $e_a$ is cyclic, the cyclic form of the spectral theorem identifies
$\ell^2(m)$ unitarily with the spectral space
\[
  L^2(\psi_a),
  \qquad
  \psi_a(B)=\langle e_a,E(B)e_a\rangle,
  \quad B\in\mathcal B([0,\infty)),
\]
under which $A$ becomes multiplication by $x$ and each normalized
coordinate vector $e_i$ becomes a function
$\phi_i\in L^2(\psi_a)$. Consequently,
\[
  \sqrt{\frac{m_i}{m_j}}\,p_{ij}(t)
  =
  \int_{[0,\infty)}e^{-xt}
  \phi_i(x)\phi_j(x)\,\psi_a(\dd x).
\]
Writing
$Q_i(x)={\phi_i(x)}/{\sqrt{m_i}}$
gives the scalar representation
\begin{equation}
\label{eq:scalar-full}
  p_{ij}(t)
  =
  m_j\int_{[0,\infty)}e^{-xt}
  Q_i(x)Q_j(x)\,\psi_a(\dd x).
\end{equation}

If every power of $A$ is defined at $e_a$ and
$\overline{\operatorname{span}} \{e_a,Ae_a,A^2e_a,\ldots\}$
equals the cyclic subspace generated by $e_a$, Gram--Schmidt
orthogonalization produces a basis in which $A$ has tridiagonal form.
Under the spectral representation, the corresponding basis functions
are orthogonal polynomials for $\psi_a$.

If $e_a$ is not cyclic, the scalar spectral measure $\psi_a$ describes
only the reducing subspace generated by $e_a$, rather than the whole
operator. For example, on two identical branches joined at the state
$a$, antisymmetric functions vanish at $a$ and are therefore not
detected by $e_a$.

If no single coordinate vector is cyclic, the spectral representation
need not be scalar. The general form of the spectral theorem represents
$A$ as multiplication by $x$ on a space of vector-valued functions.
More precisely, there are a measure $\psi$ and Hilbert spaces
$\mathcal K_x$, defined for $\psi$-almost every $x$, such that each
normalized coordinate vector $e_i$ is represented by a vector
$\phi_i(x)\in\mathcal K_x$ and
\[
  \sqrt{\frac{m_i}{m_j}}\,p_{ij}(t)
  =
  \int_{[0,\infty)}e^{-xt}
  \langle\phi_i(x),\phi_j(x)\rangle_{\mathcal K_x}
  \,\psi(\dd x).
\]
Writing
$\Phi_i(x)=\phi_i(x)/\sqrt{m_i}$ gives
\begin{equation}
\label{eq:fibre}
  p_{ij}(t)
  =
  m_j\int_{[0,\infty)}e^{-xt}
  \langle\Phi_i(x),\Phi_j(x)\rangle_{\mathcal K_x}
  \,\psi(\dd x).
\end{equation}
The dimension of $\mathcal K_x$ records the spectral multiplicity at
$x$. In the scalar case, $\dim\mathcal K_x=1$ almost everywhere; higher
dimensions allow several independent spectral components to contribute
at the same spectral value.

The Anderson--McDunnough representation \eqref{eq:AM-multi} gives a
concrete version of this vector-valued picture. The seed states index a
matrix-valued spectral measure $(\gamma_{\alpha\beta})_{\alpha,\beta\in
E_0}$. After choosing a common dominating measure, its Radon--Nikodym
derivative is a positive semidefinite matrix almost everywhere.
Factoring that matrix expresses the integrand in~\eqref{eq:AM-multi} as
an inner product of vectors associated with $i$ and $j$. Thus the cross
terms are retained in the vector-valued representation~\eqref{eq:fibre}.

Because the spaces $\mathcal K_x$ and the representing vectors are
defined only for $\psi$-almost every $x$, an expression such as that for
$\Phi_i(\lambda)$ should be understood as limiting edge data obtained
from suitable representatives, when such a limit exists.

\medskip\noindent{\bf Consequences for polynomial decay and survival.}\quad 
Theorem~\ref{thm:bd-spectral-edge} shows that, in the
scalar Karlin--McGregor setting, regular variation at the lower endpoint
produces a common fixed-state exponent and rank-one coefficients. The
multichannel representations above suggest that rank one should instead
be understood as a multiplicity-one phenomenon at the dominant spectral
edge. Fixed-state asymptotics do not determine survival without
additional uniformity in the terminal state. The following elementary
proposition identifies the tightness condition needed to justify
summation of the fixed-state limits.

\begin{proposition}[Tightness criterion for survival asymptotics]
\label{prop:tightness-survival}
Fix $i\in C$, and let $a(t)>0$ for all sufficiently large $t$.
Suppose that, for every fixed $j\in C$,
\begin{equation}
\label{eq:tightness-pointwise-limit}
 \frac{p_{ij}(t)}{a(t)}
 \to
 c_{ij},
 \qquad t\to\infty,
\end{equation}
where
$c_{ij}\geq 0$ and $\sum_{j\in C}c_{ij}<\infty$.
Suppose also that the normalized transition measures are
asymptotically tight, in the sense that
\begin{equation}
\label{eq:asymptotic-tightness}
 \inf_{\substack{F\subset C\\ F\text{ finite}}}
 \limsup_{t\to\infty}
 \sum_{j\in C\setminus F}
 \frac{p_{ij}(t)}{a(t)}
 =
 0.
\end{equation}
Then
\begin{equation}
\label{eq:survival-from-tightness}
 \frac{\mathbb P_i(T_0>t)}{a(t)}
 =
 \sum_{j\in C}\frac{p_{ij}(t)}{a(t)}
 \to
 \sum_{j\in C}c_{ij}.
\end{equation}
\end{proposition}

\begin{proof}
For every finite set $F\subset C$, pointwise convergence and
finiteness of $F$ give
\[
 \sum_{j\in F}\frac{p_{ij}(t)}{a(t)} \to \sum_{j\in F}c_{ij}.
\]
Since all terms are nonnegative,
\[
 \liminf_{t\to\infty}
 \sum_{j\in C}\frac{p_{ij}(t)}{a(t)}
 \geq
 \sum_{j\in F}c_{ij}.
\]
Letting $F$ increase through finite subsets of $C$ yields
\begin{equation}
\label{eq:tightness-lower-bound}
 \liminf_{t\to\infty}
 \sum_{j\in C}\frac{p_{ij}(t)}{a(t)}
 \geq
 \sum_{j\in C}c_{ij}.
\end{equation}

For the upper bound, let $\varepsilon>0$. By
\eqref{eq:asymptotic-tightness}, there is a finite set $F_1\subset C$
such that
\[
 \limsup_{t\to\infty}
 \sum_{j\in C\setminus F_1}
 \frac{p_{ij}(t)}{a(t)}
 <\varepsilon.
\]
Since $\sum_jc_{ij}<\infty$, there is a finite set
$F_2\subset C$ such that
$\sum_{j\in C\setminus F_2}c_{ij}<\varepsilon$.
Put $F=F_1\cup F_2$. Then
\[
 \begin{aligned}
 \limsup_{t\to\infty}
 \sum_{j\in C}\frac{p_{ij}(t)}{a(t)}
 &\leq
 \lim_{t\to\infty}
 \sum_{j\in F}\frac{p_{ij}(t)}{a(t)}
 +
 \limsup_{t\to\infty}
 \sum_{j\in C\setminus F}
 \frac{p_{ij}(t)}{a(t)}\\
 &\leq
 \sum_{j\in F}c_{ij}+\varepsilon
 \leq
 \sum_{j\in C}c_{ij}+\varepsilon.
 \end{aligned}
\]
Letting $\varepsilon\downarrow0$, and combining the resulting upper
bound with \eqref{eq:tightness-lower-bound}, proves
\eqref{eq:survival-from-tightness}.
\end{proof}

\begin{remark}[Application to rank-one fixed-state limits]
\label{rem:rank-one-survival-tightness}
Suppose that
$e^{\lambda t}p_{ij}(t) \sim A x_i m_j t^{-\kappa}L(t)$
for every fixed $i,j\in C$, where
$\sum_{j\in C}m_j<\infty$.
Taking
$a(t)=e^{-\lambda t}t^{-\kappa}L(t)$
in Proposition~\ref{prop:tightness-survival}, we see that asymptotic
tightness of the measures
\[
 \nu_{i,t}(\{j\}) = \frac{e^{\lambda t}p_{ij}(t)} {t^{-\kappa}L(t)}
\]
implies
\[
e^{\lambda t}\mathbb P_i(T_0>t) \sim A x_i \left(\sum_{j\in C}m_j\right) t^{-\kappa}L(t).
\]
A summable dominating family is a convenient sufficient condition for
asymptotic tightness, but asymptotic tightness is the more intrinsic
condition.
\end{remark}

\begin{remark}[Consequences for limiting conditional distributions]
\label{rem:lcd-tightness}
Under the assumptions of
Proposition~\ref{prop:tightness-survival}, if
$\sum_{k\in C}c_{ik}>0$, then
\[
 \Pp_i(X(t)=j\mid T_0>t)
 \to
 \frac{c_{ij}}{\sum_{k\in C}c_{ik}},
 \qquad j\in C.
\]
Thus summability of the limiting coefficients together with
asymptotic tightness yields a proper LCD. In the rank-one case
$c_{ij}=Ax_i m_j$, this becomes ${m_j}/{\sum_{k\in C}m_k}$,
and is independent of the initial state. Without asymptotic tightness,
the pointwise conditional probabilities may all converge to $0$, as
in the hub-and-traps and homogeneous-tree examples.
\end{remark}

\begin{remark}[Escape of normalized mass]
\label{rem:escape-normalized-mass}
The hub-and-traps model shows that summability of the limiting left
factor does not imply asymptotic tightness. When $\kappa>2$, its
coefficient sequence $m$ is summable, but normalized mass escapes to
traps whose indices are of order $t$, and the survival exponent is
$\kappa-1$, rather than $\kappa$.
The homogeneous-tree example exhibits a stronger failure at the
fixed-state normalization. Its fixed-state limits are of order
$e^{-\lambda t}t^{-3/2}$, whereas
$\mathbb P_x(T_0>t)=e^{-\delta t}$
with $\delta<\lambda$. Consequently, the total masses of the
normalized transition measures grow exponentially and are not
asymptotically tight.
\end{remark}

We now return to the fixed-state asymptotics and show that, when the
contributing channels share a common regularly varying edge scale, their
leading contributions combine into an explicit coefficient kernel.

\begin{proposition}[Common-scale multichannel edge asymptotics]
\label{prop:multichannel-edge}
Suppose that the seed-set representation~\eqref{eq:AM-multi}
holds, with cross-spectral measures
$\gamma_{\alpha\beta}$, $\alpha,\beta\in E_0$.
Suppose that these measures are supported on
$[\lambda,\infty)$ and that, for some $\kappa>0$ and
slowly varying function $L$, there are real constants
$c_{\alpha\beta}$ such that
\begin{equation}
\label{eq:matrix-edge-law}
 \gamma_{\alpha\beta}([\lambda,\lambda+y])
 =
 \bigl(c_{\alpha\beta}+o(1)\bigr)
 y^\kappa L(1/y),
 \qquad y\downarrow0.
\end{equation}
Suppose also that the generalized eigenfunctions
$Q_i^\alpha$ are continuous at $\lambda$ and that,
for each fixed $i,j,\alpha,\beta$, there is a $t_0>0$
such that
\[
 \int_{[\lambda,\infty)}
 e^{-t_0x}|Q_i^\alpha(x)Q_j^\beta(x)|
 \,|\gamma_{\alpha\beta}|(\mathrm dx)<\infty.
\]
Then, for every fixed $i,j\in C$,
\[
 e^{\lambda t}p_{ij}(t)
 =
 C_{ij}t^{-\kappa}L(t)
 +o\bigl(t^{-\kappa}L(t)\bigr),
\]
where
\begin{equation}
\label{eq:multichannel-limit}
 C_{ij}
 =
 \Gamma(\kappa+1)m_j
 \sum_{\alpha,\beta\in E_0}
 \frac{c_{\alpha\beta}}
      {\sqrt{m_\alpha m_\beta}}\,
 Q_i^\alpha(\lambda)Q_j^\beta(\lambda).
\end{equation}
For each fixed $i,j$, the sum is finite.
The matrix $(c_{\alpha\beta})$ is positive
semidefinite, as is the symmetrized coefficient
kernel
$\widetilde C_{ij}
 =C_{ij}\sqrt{m_i/m_j}$.
If $C_{ij}>0$, then
$e^{\lambda t}p_{ij}(t)
 \sim C_{ij}t^{-\kappa}L(t)$.
If $C_{ij}=0$, the conclusion is only
$e^{\lambda t}p_{ij}(t)
 =o(t^{-\kappa}L(t))$.
\end{proposition}

\begin{proof}
Fix $\alpha,\beta\in E_0$, and shift
$\gamma_{\alpha\beta}$ by $\lambda$ to obtain
a measure $\widetilde\gamma_{\alpha\beta}$
on $[0,\infty)$. Put
$U_{\alpha\beta}(y)
 =\widetilde\gamma_{\alpha\beta}([0,y])$.
The diagonal edge laws, together with $\kappa>0$,
exclude atoms at $0$ for the diagonal measures.
The spectral Cauchy--Schwarz inequality then
excludes an atom there for each cross measure.

Integration by parts gives
\[
 \int_{[0,\infty)}e^{-ty}
 \widetilde\gamma_{\alpha\beta}(\mathrm dy)
 =
 t\int_0^\infty e^{-ty}
 U_{\alpha\beta}(y)\,\mathrm dy.
\]
Using~\eqref{eq:matrix-edge-law}, followed by
the substitution $u=ty$ and Potter bounds,
we obtain
\begin{equation}
\label{eq:cross-edge-laplace}
 \int_{[0,\infty)}e^{-ty}
 \widetilde\gamma_{\alpha\beta}(\mathrm dy)
 =
 c_{\alpha\beta}\Gamma(\kappa+1)
 t^{-\kappa}L(t)
 +o\bigl(t^{-\kappa}L(t)\bigr).
\end{equation}
The additive form includes the possibility
$c_{\alpha\beta}=0$.

By continuity,
$Q_i^\alpha(\lambda+y)Q_j^\beta(\lambda+y)$
converges to
$Q_i^\alpha(\lambda)Q_j^\beta(\lambda)$
as $y\downarrow0$. The total variation of a
cross-spectral measure on a Borel set $B$
satisfies
\[
 |\gamma_{\alpha\beta}|(B)
 \leq
 \sqrt{\gamma_{\alpha\alpha}(B)
       \gamma_{\beta\beta}(B)}.
\]
To make this localization explicit, put
$R(y)=Q_i^\alpha(\lambda+y)Q_j^\beta(\lambda+y)$.
For sufficiently small $\eta>0$,
\[
 \left|
 \int_{[0,\eta]}e^{-ty}\{R(y)-R(0)\}
 \,\widetilde\gamma_{\alpha\beta}(\mathrm dy)
 \right|
 \leq
 \sup_{0\leq y\leq\eta}|R(y)-R(0)|
 \int_{[0,\eta]}e^{-ty}
 \,|\widetilde\gamma_{\alpha\beta}|(\mathrm dy).
\]
The total-variation bound above and the diagonal edge laws imply
\[
 |\widetilde\gamma_{\alpha\beta}|([0,y])
 =O\bigl(y^\kappa L(1/y)\bigr)
 \qquad (y\downarrow0).
\]
Integration by parts and Potter bounds therefore give
\[
 \int_{[0,\eta]}e^{-ty}
 \,|\widetilde\gamma_{\alpha\beta}|(\mathrm dy)
 =O\bigl(t^{-\kappa}L(t)\bigr).
\]
By continuity of $R$ at $0$, first letting $t\to\infty$
and then $\eta\downarrow0$ shows that the replacement error
is $o(t^{-\kappa}L(t))$. On the complement of a fixed edge
neighbourhood, the stated integrability condition makes the
integral exponentially small.
It follows from~\eqref{eq:cross-edge-laplace}
that
\[
 \int_{[\lambda,\infty)}
 e^{-xt}Q_i^\alpha(x)Q_j^\beta(x)
 \,\gamma_{\alpha\beta}(\mathrm dx)
  =
 e^{-\lambda t}
 \left\{
 c_{\alpha\beta}\Gamma(\kappa+1)
 Q_i^\alpha(\lambda)Q_j^\beta(\lambda)
 t^{-\kappa}L(t)
 +o\bigl(t^{-\kappa}L(t)\bigr)
 \right\}.
\]
For fixed $i,j$, only finitely many pairs
$(\alpha,\beta)$ contribute to
\eqref{eq:AM-multi}. Summing these expansions
proves~\eqref{eq:multichannel-limit}.

For every finite $F\subset E_0$ and real
$(a_\alpha:\alpha\in F)$, positivity of the
spectral projection gives
\[
 \sum_{\alpha,\beta\in F}
 a_\alpha a_\beta
 \gamma_{\alpha\beta}
 ([\lambda,\lambda+y])\geq 0.
\]
Divide by $y^\kappa L(1/y)$ and let
$y\downarrow0$. Hence
$(c_{\alpha\beta})$ is positive semidefinite.
Formula~\eqref{eq:multichannel-limit} then
shows that $\widetilde C$ is positive
semidefinite. Finally, $p_{ij}(t)\geq 0$
implies $C_{ij}\geq 0$, and the two stated
alternatives follow.
\end{proof}

The coefficient kernel is governed by the \emph{matrix}
$(c_{\alpha\beta})$, not by independent diagonal seed channels. In fact,
its algebraic rank equals the rank of this edge matrix. To see one
direction, note that \eqref{eq:multichannel-limit} factors $C$ through
$(c_{\alpha\beta})$. For the other, restrict $i$ and $j$ to seed states
and use $Q_\beta^\alpha(\lambda)=\delta_{\alpha\beta}$: the resulting
submatrix of $C$ is a diagonally rescaled copy of $(c_{\alpha\beta})$.
Thus rank one holds precisely when the edge matrix has rank one. Under
irreducibility, if the edge matrix is nonzero, all the coefficients
$C_{ij}$ are strictly positive. Indeed, a nonzero positive-semidefinite
edge matrix has $c_{aa}>0$ for some seed state $a$, so $C_{aa}>0$; the
Chapman--Kolmogorov comparisons then give $C_{ij}>0$ for every fixed
$i,j\in C$.

If different channels have different edge exponents, the leading scale
for a fixed pair is determined by the smallest visible exponent, unless
the corresponding contributions cancel. Thus a common fixed-state
exponent requires every pair to couple nontrivially to the same dominant
edge scale, while rank one further requires the associated edge vectors
to span a one-dimensional subspace.

Survival remains separate. Passing from fixed-state limits to
$\Pp_i(T_0>t)$ requires control of the normalized mass outside finite
sets; Proposition~\ref{prop:tightness-survival} identifies asymptotic
tightness as the required condition.

\section{Killed random walk on a homogeneous tree}
\label{sec:tree-counterexample}

The multichannel discussion suggests that rank-one asymptotics should
be expected only when the dominant spectral edge has effective
multiplicity one. To test this idea we now examine a highly symmetric
irreducible example whose limiting edge coefficient kernel is not
effectively one-dimensional. The resulting kernel is positive
semidefinite but not rank one.

Let $\mathbb T_{q+1}$, where $q\geq2$, be the homogeneous tree, that is,
the infinite connected graph without cycles in which every vertex has
degree $q+1$. Its vertex set is countably infinite and may therefore be
identified with $\{1,2,\ldots\}$. Add an absorbing state $0$, and write
$S=\{0\}\cup C$, $C=\mathbb T_{q+1}$. From each $x\in C$, the process
jumps to each neighbouring vertex at rate $1$ and is killed into~$0$ at
rate $\delta>0$. Thus, writing $x\sim y$ when the vertices $x$ and $y$
are adjacent,
\[
 q_{xy}
 =
 \begin{cases}
  1, & x,y\in C,\ x\sim y,\\
  \delta, & x\in C,\ y=0,\\
  -(q+1+\delta), & x=y\in C,\\
  0, & \text{otherwise}.
 \end{cases}
\]
The class $C$ is irreducible because the tree is connected. Moreover,
$\sup_{x\in C}q_x=q+1+\delta<\infty$. Thus $Q$ is regular
and the minimal process is nonexplosive. 
Killing occurs at an independent exponential time with rate $\delta$, and hence absorption occurs almost surely from every $x\in C$.

Let $\mathcal A$ denote the adjacency operator of~$\mathbb T_{q+1}$, acting
on $\ell^2(C)$, and let
$L=(q+1)I-\mathcal A$.
The unkilled transition semigroup is $e^{-tL}$. If $h_t(x,y)$ denotes its transition kernel, then the killed transition probabilities are
\begin{equation}
\label{eq:tree-killed-kernel}
 p_{xy}(t)=e^{-\delta t}h_t(x,y),
 \qquad x,y\in C.
\end{equation}
We first record the spectral information needed for the asymptotics.
The spherical spectral representation and the associated root spectral
measure are standard; see for example~\cite{Woess2000}. We include a
derivation to fix the normalization.

\begin{lemma}[Spherical spectral representation]
\label{lem:tree-spectral-representation}
Fix a vertex $o\in C$. The spectral measure of~$\mathcal A$ associated with
$\mathbf 1_{\{o\}}$ is supported on
$[-2\sqrt q,2\sqrt q]$ and has density
\begin{equation}
\label{eq:kesten-mckay-density}
 w_q(s)
 =
 \frac{q+1}{2\pi}
 \frac{\sqrt{4q-s^2}}
      {(q+1)^2-s^2}
 \mathbf 1_{\{|s|<2\sqrt q\}}.
\end{equation}
For every $x\in C$, with $r=d(o,x)$,
\begin{equation}
\label{eq:tree-spherical-integral}
 h_t(o,x)
 =
 e^{-(q+1)t}
 \int_{-2\sqrt q}^{2\sqrt q}
 e^{st}P_r(s)w_q(s)\,\mathrm ds,
\end{equation}
where $P_r$ is the spherical polynomial determined by
$P_0(s)=1$, $P_1(s)=s/(q+1)$, and
$sP_r(s)=P_{r-1}(s)+qP_{r+1}(s)$, $r\geq1$.
At the upper spectral endpoint,
\begin{equation}
\label{eq:tree-edge-spherical-function}
 P_r(2\sqrt q)
 =
 \varphi_0(r)
 :=
 q^{-r/2}
 \left(
  1+\frac{q-1}{q+1}r
 \right).
\end{equation}
\end{lemma}

\begin{proof}
We first determine the spectral measure at a fixed vertex from its
resolvent. Homogeneity implies that this measure is independent of the
chosen vertex, while the tree structure reduces the resolvent
calculation to a scalar recursion on a forward branch.

Fix a vertex $o\in C$ and write
\[
  G_o(z)
  =
  \left\langle
    \mathbf 1_{\{o\}},
    (zI-\mathcal A)^{-1}\mathbf 1_{\{o\}}
  \right\rangle.
\]
This is the diagonal resolvent at $o$. If $H(z)$ denotes the corresponding
diagonal resolvent at the root of a forward branch obtained by removing
one edge, then removal of the branch root gives
\[
  H(z)=\frac{1}{z-qH(z)}.
\]
The solution satisfying $H(z)\sim z^{-1}$ as $z\to\infty$ is
\[
 H(z) = \frac{z-\sqrt{z^2-4q}}{2q},
\]
where the square root is chosen so that
$\sqrt{z^2-4q}\sim z$ at infinity. At the root,
\[
 G_o(z) = \frac{1}{z-(q+1)H(z)} = \frac{2q} {(q-1)z+(q+1)\sqrt{z^2-4q}}.
\]
For $q\geq2$, the displayed resolvent has no real poles
outside $[-2\sqrt q,2\sqrt q]$, with the square-root branch
specified above. It is analytic across every real interval
disjoint from this set. 
Hence the root spectral measure has no mass outside
$[-2\sqrt q,2\sqrt q]$. There are no atoms at the endpoints
either. Indeed, the displayed formula for $G_o(z)$ remains bounded
as $z$ approaches either $2\sqrt q$ or $-2\sqrt q$ from the upper
half-plane, since its denominator has a nonzero limit there for
$q\geq2$. An atom at an endpoint would make the corresponding
resolvent unbounded.
The Stieltjes inversion formula recovers the density of the spectral
measure from the imaginary part of the boundary value of its resolvent.
Taking $z=s+\mathrm i0$, with $|s|<2\sqrt q$, therefore gives
\[
 -\frac{1}{\pi}\operatorname{Im}G_o(s+\mathrm i0) = \frac{q+1}{2\pi} \frac{\sqrt{4q-s^2}} {(q+1)^2-s^2},
\]
which proves \eqref{eq:kesten-mckay-density}.

The radial subspace generated by $\mathbf 1_{\{o\}}$ is invariant
under $\mathcal A$. The generalized radial eigenfunction normalized to equal
$1$ at $o$ satisfies
$P_0(s)=1$, $(q+1)P_1(s)=sP_0(s)$,
and, away from the root,
$P_{r-1}(s)+qP_{r+1}(s)=sP_r(s)$.
The spectral theorem then gives
\eqref{eq:tree-spherical-integral}.

At $s=2\sqrt q$, the recurrence has a repeated characteristic root
$q^{-1/2}$. Hence
$P_r(2\sqrt q) = q^{-r/2}(a+br)$
for constants $a,b$. The condition $P_0=1$ gives $a=1$, while
$P_1(2\sqrt q)=2\sqrt q/(q+1)$ gives $b=(q-1)/(q+1)$.
This proves \eqref{eq:tree-edge-spherical-function}.
\end{proof}

The long-time behaviour of the heat kernel is determined by the spectral
density near its upper endpoint. The preceding representation therefore
reduces the required asymptotics to an endpoint Laplace calculation.

\begin{theorem}[Fixed-state heat-kernel asymptotic]
\label{thm:tree-heat-kernel-asymptotic}
For every fixed $x,y\in C$,
\begin{equation}
\label{eq:tree-rigorous-asymptotic}
 h_t(x,y)
 \sim
 C_q\varphi_0(d(x,y))
 e^{-(q+1-2\sqrt q)t}t^{-3/2},
 \qquad t\to\infty,
\end{equation}
where
\begin{equation}
\label{eq:tree-constant}
 C_q
 =
 \frac{(q+1)q^{1/4}}
 {2\sqrt{\pi}(q-1)^2}
\end{equation}
and
\[
 \varphi_0(r) = q^{-r/2} \left( 1+\frac{q-1}{q+1}r \right).
\]
\end{theorem}

\begin{proof}
By homogeneity, it is enough to prove the result for $x=o$ and
$r=d(o,y)$. Put
$s=2\sqrt q-u$.
As $u\downarrow0$,
$\sqrt{4q-s^2} = \sqrt{4\sqrt q\,u-u^2} \sim 2q^{1/4}u^{1/2}$,
and
$(q+1)^2-s^2 \to (q+1)^2-4q = (q-1)^2$.
It follows from \eqref{eq:kesten-mckay-density} that
\begin{equation}
\label{eq:tree-density-edge}
 w_q(2\sqrt q-u)
 \sim
 D_q u^{1/2},
 \qquad
 D_q
 =
 \frac{(q+1)q^{1/4}}
 {\pi(q-1)^2}.
\end{equation}
Since $P_r$ is a polynomial,
$P_r(2\sqrt q-u) \to P_r(2\sqrt q) = \varphi_0(r)$.

Fix $\eta\in(0,2\sqrt q)$ small enough that $P_r(2\sqrt q-u)>0$ for
$0\leq u\leq\eta$. The part of the integral in
\eqref{eq:tree-spherical-integral} corresponding to $s\leq 2\sqrt q-\eta$
is $O\left( e^{-(q+1-2\sqrt q+\eta)t} \right)$, because $P_r$ is bounded
on the compact spectral interval and $w_q(s)\,\mathrm ds$ is a finite
measure.

For the remaining part, the endpoint asymptotics
\eqref{eq:tree-density-edge} and continuity of $P_r$ give
\[
 \begin{aligned}
 h_t(o,y)
 &\sim
 e^{-(q+1-2\sqrt q)t}
 D_q\varphi_0(r)
 \int_0^\eta e^{-tu}u^{1/2}\,\mathrm du\\
 &\sim
 e^{-(q+1-2\sqrt q)t}
 D_q\varphi_0(r)
 \Gamma(3/2)t^{-3/2}.
 \end{aligned}
\]
Since $\Gamma(3/2)={\sqrt\pi}/{2}$, we have
\[
 D_q\Gamma(3/2) = \frac{(q+1)q^{1/4}} {2\sqrt\pi(q-1)^2} = C_q.
\]
This proves the result.
\end{proof}

\begin{remark}[Positive semidefiniteness of the edge kernel]
\label{rem:tree-edge-positive-semidefinite}
For every finite set $F\subset C$ and every real family
$(a_x:x\in F)$, self-adjointness and positivity of the heat
semigroup give
$\sum_{x,y\in F}a_xa_yh_t(x,y)\geq 0$.
Multiplying by the positive scalar
$C_q^{-1}e^{(q+1-2\sqrt q)t}t^{3/2}$
and using Theorem~\ref{thm:tree-heat-kernel-asymptotic} term by term
over the finite set $F\times F$ yields
$\sum_{x,y\in F} a_xa_y\varphi_0(d(x,y)) \geq 0$.
Hence
$(\varphi_0(d(x,y)),\, x,y\in C)$
is a positive semidefinite kernel. Consequently, the coefficient
kernel in Theorem~\ref{thm:tree-counterexample} is also positive
semidefinite.
\end{remark}

The preceding theorem agrees, after conversion of the time
normalization, with the sharp fixed-distance asymptotics for the natural
Laplacian on a homogeneous tree obtained in
\cite{Papageorgiou2026Arxiv}. The direct spectral derivation above is
included to fix the convention that every edge is traversed at rate
$1$.

Combining the fixed-state heat-kernel asymptotics with independent
exponential killing now gives the announced counterexample: the
fixed-state probabilities have a common polynomial correction, but
their coefficient kernel is not rank one and survival occurs on a
different exponential scale.

\begin{theorem}[An irreducible $\lambda$-transient counterexample]
\label{thm:tree-counterexample}
For killed random walk on $\mathbb T_{q+1}$ defined above, put
$\lambda = \delta+q+1-2\sqrt q$.
Then the following statements hold.

\begin{enumerate}
\item The transient class $C=\mathbb T_{q+1}$ is irreducible.

\item For every fixed $x,y\in C$,
\begin{equation}
\label{eq:tree-killed-asymptotic}
 e^{\lambda t}p_{xy}(t)
 \sim
 C_q\varphi_0(d(x,y))t^{-3/2}.
\end{equation}
Thus all fixed-state transition probabilities have the common residual
exponent $3/2$.

\item The chain is $\lambda$-transient.

\item The coefficient kernel
$c_{xy} = C_q\varphi_0(d(x,y))$
does not have rank one.

\item The survival probability has exponential decay parameter
$\delta$, strictly smaller than $\lambda$. In particular, it
does not have a polynomially renormalised limit at the fixed-state
decay parameter $\lambda$.
\end{enumerate}
\end{theorem}

\begin{proof}
Irreducibility was established above from the connectedness of the
tree. Combining \eqref{eq:tree-killed-kernel} with
Theorem~\ref{thm:tree-heat-kernel-asymptotic} gives
\eqref{eq:tree-killed-asymptotic}.
Because
$e^{\lambda t}p_{xy}(t) \sim C_q\varphi_0(d(x,y))t^{-3/2}$,
and $t^{-3/2}$ is integrable at infinity,
$\int_0^\infty e^{\lambda t}p_{xy}(t)\,\mathrm dt<\infty$.
Thus the irreducible class $C$ is $\lambda$-transient.

To disprove rank-one factorisation, choose adjacent vertices
$x\sim y$. Then
\[
 c_{xx}=c_{yy}=C_q,
 \qquad
 c_{xy}=c_{yx}=C_q\varphi_0(1),
\]
where
\[
 \varphi_0(1) = q^{-1/2} \left( 1+\frac{q-1}{q+1} \right) = \frac{2\sqrt q}{q+1}.
\]
Therefore,
\[
 \det
 \begin{pmatrix}
  c_{xx}&c_{xy}\\
  c_{yx}&c_{yy}
 \end{pmatrix}
 =
 C_q^2
 \left(
  1-\frac{4q}{(q+1)^2}
 \right)
 =
 C_q^2\frac{(q-1)^2}{(q+1)^2}
 >0.
\]
Every $2\times2$ minor of a rank-one matrix is $0$. Hence
$(c_{xy})$ is not rank one.

Finally, killing occurs through an independent exponential clock with
rate $\delta$, so
$\mathbb P_x(T_0>t)=e^{-\delta t}$.
Moreover, independence of the killing clock gives
\[
 \Pp_x(X(t)=y\mid T_0>t)
 =
 h_t(x,y),
 \qquad x,y\in C.
\]
By Theorem~\ref{thm:tree-heat-kernel-asymptotic},
$h_t(x,y)\to0$ for every fixed $x,y\in C$. Hence the conditioned laws
escape to infinity and there is no proper LCD over~$C$.
Consequently,
\[
e^{\lambda t}\mathbb P_x(T_0>t) = e^{(q+1-2\sqrt q)t} = e^{(\sqrt q-1)^2t}.
\]
Since $q\geq2$, this grows exponentially. Thus no polynomial
renormalisation at the fixed-state decay parameter $\lambda$ can
produce a finite positive survival limit.
\end{proof}

Although the edge coefficient kernel is not rank one, the explicit
fixed-state asymptotics still yield a complete ratio limit.

\begin{corollary}[A fixed-state ratio limit]
For fixed $x,y,z\in C$,
\[
  \frac{p_{xy}(t)}{p_{xz}(t)}
  \to
  \frac{\varphi_0(d(x,y))}{\varphi_0(d(x,z))},
  \qquad t\to\infty.
\]
In particular,
${p_{xy}(t)}/{p_{xx}(t)}
\to \varphi_0(d(x,y))$.
\end{corollary}

\begin{proof}
This follows immediately from
\eqref{eq:tree-killed-asymptotic}, since the common exponential and
polynomial factors cancel.
\end{proof}

Although there is no proper LCD on $C$, the conditioned law has a
particularly simple description. As observed in the proof of
Theorem~\ref{thm:tree-counterexample}, conditioning on survival exactly
removes the independent killing and leaves the law of the unkilled
random walk. Thus the conditional behaviour is governed by spatial
escape rather than by convergence to a distribution on fixed states.
The absence of an LCD reflects escape of mass, not a lack of structure;
a limit based on the spatial location of the walk is the natural
analogue of the scaled conditional limit in the hub-and-traps model.

\begin{remark}[Consequences for the conjectures]
\label{rem:tree-conjecture-consequences}
The example retains a common fixed-state exponent and an explicit ratio
limit, while its positive-semidefinite edge coefficient kernel has rank
greater than one. Thus irreducibility alone does not imply rank-one
factorisation. Moreover, conditioning on survival recovers the unkilled
random walk, whose mass escapes spatially, and survival is governed by a
strictly smaller exponential decay parameter than the fixed-state
transition probabilities.
\end{remark}

\section{Discussion and open questions}\label{sec:discussion}

\noindent
\textbf{Main results and the original conjectures.}\quad
The original conjectural picture contains three separate issues. The
first is a solidarity question: how much of the asymptotic behaviour of
all fixed-state transition probabilities is determined by that of a
single entry? The second concerns the structure of the limiting
coefficient kernel. The third asks whether fixed-state asymptotics
control the total surviving mass. The results and examples above show
that positive answers to one of these questions need not give positive
answers to the others.

\medskip
\noindent
\textit{Common regular variation.}\quad
Proposition~\ref{prop:polynomial-order-solidarity} shows that the
Chapman--Kolmogorov comparisons transmit the polynomial order of one
regularly varying residual to every fixed pair of states. What they do
not provide is convergence of the corresponding ratios. Thus the
remaining content of Conjecture~\ref{conj:common} is not solidarity of
order, but the stronger assertion that
\[
  \frac{e^{\lambda t}p_{ij}(t)}
       {e^{\lambda t}p_{ab}(t)}
\]
has a positive finite limit for every fixed $i,j,a,b\in C$. The
examples considered here all have such ratio limits, as does the scalar
spectral setting of Theorem~\ref{thm:bd-spectral-edge}, but no general
argument from irreducibility alone is presently available.

\medskip
\noindent
\textit{The coefficient kernel.}\quad
The renewal and spectral calculations explain why rank one occurs in
the first three settings considered in the paper. In the killed stable
$M/M/1$ queue and the hub-and-traps model, a single dominant excursion
mechanism separates the dependence on the initial and terminal states.
For birth--death processes, the scalar Karlin--McGregor representation
produces the same separation through evaluation of the orthogonal
polynomials at the lower spectral edge.

The homogeneous-tree example identifies the limitation of this
picture. Its fixed-state probabilities still have a common asymptotic
scale and positive ratio limits, but the edge contribution retains more
than one asymptotically visible direction. The limiting coefficient
kernel is therefore positive semidefinite but not rank one. This shows
that the failure of Conjecture~\ref{conj:rankone} is not caused by a
failure of common fixed-state asymptotics; it is caused by higher
effective spectral multiplicity.

\medskip
\noindent
\textit{Survival and escape of mass.}\quad
The hub-and-traps and homogeneous-tree examples exhibit different ways
in which fixed-state limits can fail to determine survival. In the
hub-and-traps model, fixed-state probabilities and survival have the
same exponential decay parameter, but survival has the heavier
polynomial residual $t^{-(\kappa-1)}$. The missing mass is carried by
traps whose indices grow with $t$. This is compatible with
Conjecture~\ref{conj:survival}, which allows
$\kappa_0<\kappa$, but it shows why termwise summation of fixed-state
asymptotics is invalid.

On the homogeneous tree, the separation is stronger: survival has
exponential decay parameter $\delta<\lambda$. This rules out the
survival asymptotic proposed in Conjecture~\ref{conj:survival} at the
fixed-state exponential scale.
Proposition~\ref{prop:tightness-survival} locates the obstruction in
both examples: passage from pointwise fixed-state limits to survival
requires asymptotic control of the mass outside finite subsets of the
state space.

The same distinction governs conditional limits. The killed stable
$M/M/1$ queue has a proper LCD. In the hub-and-traps model the
conditioned state is typically of order $t$, and a nondegenerate limit
appears only after spatial rescaling. On the homogeneous tree,
conditioning removes the independent killing and leaves the spatially
escaping unkilled random walk. Thus the absence of a proper LCD in the
last two examples is an escape-of-mass phenomenon, although the form of
that escape is different in the two models.

\medskip
\noindent
\textbf{Refined spectral conjectures.}\quad
The two conjectures below do not replace all three original
conjectures. Rather, they refine the fixed-state coefficient question
in spectral terms.
Conjecture~\ref{conj:multiplicity-one} describes a setting in which the
rank-one conclusion of Conjecture~\ref{conj:rankone} should hold, while
Conjecture~\ref{conj:general-multichannel} describes the
positive-semidefinite kernel expected when several edge channels
contribute. The general ratio-limit problem in
Conjecture~\ref{conj:common} remains open, and survival remains a
separate problem involving tightness and possible escape of mass.

In the spectral representations \eqref{eq:scalar-full},
\eqref{eq:fibre}, and~\eqref{eq:multichannel-limit}, the leading
coefficients are inner products of edge vectors. Their effective
multiplicity is therefore the dimension of the span of those vectors,
with multiplicity one corresponding to rank one. This motivates the
following formulations.

\begin{conjecture}[Multiplicity-one spectral version]
\label{conj:multiplicity-one}
Let $P_C(t)$ be the transition function of an irreducible reversible
evanescent Markov chain, and let $\lambda$ be the lower endpoint of its
spectral representation. Suppose that, near $\lambda$, the dominant
spectral contribution has the scalar form of
\eqref{eq:scalar-full}, or, more generally, that the asymptotically
contributing edge vectors span a one-dimensional subspace. Suppose
also that every coordinate state couples nontrivially to this
subspace and that the scalar edge measure is regularly varying with
exponent $\kappa>0$. Under suitable continuity and localization
conditions at the edge, there is a slowly varying function $L$ such
that
\[
  e^{\lambda t}p_{ij}(t)
  \sim
  c_{ij}t^{-\kappa}L(t),
  \qquad i,j\in C,
\]
where the coefficients $c_{ij}$ are strictly positive and have rank
one.
\end{conjecture}

As in Theorem~\ref{thm:bd-spectral-edge}, the stronger condition
$\kappa>1$ is needed only if $\lambda$-transience is to follow
automatically from the fixed-state asymptotics.

\begin{conjecture}[General multichannel version]
\label{conj:general-multichannel}
Let $P_C(t)$ be the transition function of an irreducible reversible
evanescent Markov chain, reversible with respect to a strictly positive
measure $m$, and suppose that it has a vector-valued spectral
representation of the form~\eqref{eq:fibre}. Suppose that the
asymptotically contributing spectral components have a common lower
endpoint $\lambda$ and a common regularly varying edge scale
$y^\kappa L(1/y)$, $y\downarrow 0$,
for some $\kappa>0$ and slowly varying function $L$. Suppose further
that the state-dependent spectral data have limiting edge vectors
$V_i$ in a Hilbert space $\mathcal K_{\mathrm{edge}}$ formed from the
contributing components, with the edge constants absorbed into their
normalization. Under suitable continuity and localization conditions
at the edge,
\[
  t^\kappa L(t)^{-1}e^{\lambda t}p_{ij}(t)
  \to
  m_j\langle V_i,V_j\rangle_{\mathcal K_{\mathrm{edge}}},
  \qquad i,j\in C.
\]
Equivalently, the symmetrized limiting coefficient kernel satisfies
\[
  \sqrt{\frac{m_i}{m_j}}\,
  t^\kappa L(t)^{-1}e^{\lambda t}p_{ij}(t)
  \to
  \left\langle
    \sqrt{m_i}\,V_i,
    \sqrt{m_j}\,V_j
  \right\rangle_{\mathcal K_{\mathrm{edge}}}.
\]
The latter kernel is positive semidefinite, and its rank is the
dimension of
\[
  \operatorname{span}
  \{\sqrt{m_i}\,V_i:i\in C\}.
\]
In particular, the limiting coefficient kernel has rank one precisely
when these vectors span a one-dimensional subspace.
\end{conjecture}

If the limiting coefficient kernel is nonzero, irreducibility forces
$\langle V_i,V_j\rangle_{\mathcal K_{\mathrm{edge}}}>0$ for every
$i,j\in C$, by the Chapman--Kolmogorov comparisons.

Proposition~\ref{prop:multichannel-edge} gives a version of this
conclusion for the seed-set representation~\eqref{eq:AM-multi}
under its stated common-edge and localization hypotheses.

\medskip
\noindent
\textbf{Directions for further work.}\quad
The common fixed-state exponent appears to be the most robust part of
the original conjectural picture.
Proposition~\ref{prop:polynomial-order-solidarity} shows that, if one
residual $g_{ab}$ is regularly varying, every fixed-state residual has
the same logarithmic polynomial order and is comparable with the same
regularly varying function. What remains open is whether these
comparisons can be sharpened to positive ratio limits,
\[
  \frac{g_{ij}(t)}{g_{ab}(t)}
  \to c_{ij}\in(0,\infty),
  \qquad i,j\in C.
\]
The entrance and exit decomposition~\eqref{eq:full-renewal} suggests
seeking conditions on the entrance, exit, and taboo contributions that
yield such limits. The killed $M/M/1$ calculation indicates how this
can occur when one return-cycle mechanism dominates, but a general
criterion remains to be found.

A related problem is to characterize effective spectral multiplicity
without assuming a spectral representation that already displays its
channels. In the reversible setting, one would like intrinsic
conditions ensuring that the asymptotically contributing edge vectors
span a one-dimensional subspace. For the seed-set construction of
Section~\ref{sec:multichannel}, this amounts to asking when the
different seed channels become asymptotically proportional at the
lower edge. The homogeneous-tree example shows that irreducibility
alone cannot enforce this reduction.

The homogeneous-tree example also points towards a broader geometric
question. Its edge coefficient kernel is positive semidefinite but not
rank one, and its survival probability is governed by a strictly
smaller exponential decay parameter. It would be informative to know
how widely this phenomenon occurs on graphs of exponential volume
growth, and whether it is tied to nonamenability, to the multiplicity
of boundary directions, or to a more direct feature of the lower
spectral edge. Examples on less symmetric graphs could help distinguish
geometric multiplicity from the special harmonic analysis available on
a homogeneous tree.

There is also scope for constructing examples in the scalar setting.
Starting from a spectral measure with prescribed behaviour
$x^{\kappa-1}L(1/x)$ near the lower endpoint and recovering the
associated three-term recurrence could produce birth--death processes
with a range of specified polynomial exponents. The main additional
step would be to determine when the recurrence coefficients can be
normalized as strictly positive birth and death rates, with the desired
boundary behaviour. Such examples would test how much of
Theorem~\ref{thm:bd-spectral-edge} is reflected directly in the rates
and might provide useful intermediate models between the explicit
$M/M/1$ queue and the abstract spectral-edge theorem.

Survival presents a separate problem.
Proposition~\ref{prop:tightness-survival} identifies an
asymptotic-tightness condition under which fixed-state limits may be
summed. It remains to find usable hypotheses implying that condition,
for example through Lyapunov bounds, excursion estimates, or spectral
estimates uniform in the terminal state. When tightness fails, the
natural question is instead to identify the spatial scale of the
escaping conditional mass, as in the hub-and-traps conditional limit.

\medskip
\noindent
\textbf{Concluding perspective.}\quad
Two complementary pictures emerge from the results of this paper.
The renewal picture says that a long-lived or late-returning path is
often composed of an ordinary entrance segment, one exceptional
excursion, and an ordinary exit segment. When one excursion mechanism
dominates, this explains both a common polynomial factor and rank-one
coefficients. The spectral picture says that polynomial decay is the
Laplace-transform image of power-law spectral mass near the decay
parameter. In a scalar spectral representation, evaluation of the
generalized eigenfunctions at the edge gives rank one automatically.
In a multichannel representation, the symmetrized edge coefficient is
instead a positive-semidefinite kernel formed from inner products of
edge vectors, and its rank records the number of asymptotically visible
edge directions.

These pictures are compatible. A single dominant long-excursion
mechanism corresponds naturally to one effective edge channel, whereas
several geometrically distinct long-time mechanisms may appear
spectrally as higher edge multiplicity. This leads to the central
structural question:
\begin{quote}
When does irreducibility reduce the asymptotically dominant long-time
mechanism to one effective channel?
\end{quote}
The killed stable $M/M/1$ queue, the hub-and-traps model, and the scalar
birth--death theorem exhibit such a reduction; the homogeneous-tree
example does not.

Survival is governed by a different issue. Passing from fixed-state
limits to survival asymptotics at the same scale requires control of
the mass over the entire state space. When that control fails,
conditioned mass may escape spatially even though every fixed-state
probability has a precise asymptotic equivalent.

\section*{Acknowledgements}

I thank Thomas Taimre for an early observation concerning the
Gram--Schmidt route from a scalar spectral representation to
orthogonal polynomials. I am grateful to the anonymous referee for a
careful and constructive report, and particularly for suggestions
that helped clarify the status of the conjectures and improve the
organization of the concluding discussion.

\bigskip
\noindent
\textbf{Use of AI-assisted technologies.}\quad
The author used Microsoft Copilot, a GPT-based generative AI tool,
to assist with phrasing, exposition, and preliminary checks of
intermediate calculations. All mathematical arguments, calculations,
citations, and conclusions were independently reviewed and verified
by the author, who takes full responsibility for the content of the
paper.

\end{document}